\documentclass[a4paper]{amsart}
\usepackage{a4wide}

\usepackage[english]{babel}
\usepackage{amsfonts,amsmath,amssymb,amsthm,mathrsfs}
\usepackage{mathtools}
\usepackage[dvipsnames]{xcolor}
\usepackage{comment}
\usepackage{hyperref}
\usepackage{soul}

\newcommand{\bv}{\mathbf{v}}
\newcommand{\bx}{\mathbf{x}}
\newcommand{\R}{\mathbb{R}}

\renewcommand{\div}{\operatorname{div}}

\definecolor{wineRed}{rgb}{0.7,0,0.3}
\newcommand{\PR}[1]{{\color{wineRed}#1}}

\theoremstyle{plain}
\newtheorem{theorem}{Theorem}[section]
\newtheorem{lemma}[theorem]{Lemma}

\newtheorem{proposition}[theorem]{Proposition}
\newtheorem{corollary}[theorem]{Corollary}
\theoremstyle{remark}
\newtheorem{remark}[theorem]{Remark}
\theoremstyle{definition}
\newtheorem{definition}[theorem]{Definition}
\numberwithin{equation}{section}

\newcommand{\cA}{\mathcal A}
\newcommand{\cB}{\mathcal B}
\newcommand{\cE}{\mathcal E}
\newcommand{\cF}{\mathcal F}

\newcommand{\cL}{\mathcal L}
\newcommand{\cM}{\mathcal M}
\newcommand{\cS}{\mathcal S}
\newcommand{\cW}{\mathcal W}
\newcommand{\bC}{\mathbb C}
\newcommand{\bR}{\mathbb R}
\newcommand{\bN}{\mathbb N}

\newcommand{\de}{\mathrm{d}}
\newcommand{\rmH}{\mathrm{H}}
\newcommand{\rms}{\mathrm{s}}
\newcommand{\rmu}{\mathrm{u}}

\newcommand{\scrH}{\mathscr{H}}
\newcommand{\scrL}{\mathscr{L}}

\newcommand{\norm}[1]{\lVert #1 \rVert}
\newcommand{\dup}[2]{\langle #1, #2 \rangle}

\DeclareMathOperator{\dist}{dist}

\title{Stabilization of solutions to a model of Langmuir-Blodgett films}
\author[M.~Morandotti]{Marco Morandotti}
\author[P.~Rybka]{Piotr Rybka}
\author[G.~Wheeler]{Glen Wheeler}
\address[G.~Wheeler]{University of Wollongong\\
Faculty of Engineering and Information Sciences\\
Northfields Ave, NSW, Australia\\
ORCiD ID \url{https://orcid.org/0000-0003-3314-5647}}
\email{glenw@uow.edu.au}
\address[M.~Morandotti]{Politecnico di Torino Department of Mathematical Sciences ``G.~L.~Lagrange'', Corso Duca degli Abruzzi, 24, 10129 Torino, Italy. ORCiD ID \url{https://orcid.org/0000-0003-3528-6152}}
\email{marco.morandotti@polito.it}
\address[P.~Rybka]{University of Warsaw \\ 
Faculty of Mathematics, Informatics and Mechanics \\
ul. Banacha 2\\
02-097 Warsaw, Poland \\ 
ORCiD ID \url{https://orcid.org/0000-0002-0694-8201}}
\email{rybka@mimuw.edu.pl}
\date{\today}

\begin{document}

\begin{abstract}
We show stabilisation of solutions to one-dimensional advective Cahn--Hilliard equation modeling the Langmuir--Blodgett thin films. This problem has the structure of a gradient flow perturbed by a linear term  $\beta u_x$\,. Through application of an abstract result by Carvalho--Langa--Robinson, we show that for small $\beta$ the equation has the structure of gradient flow in a weak sense. Combining this with the finite number of steady states implies stabilization of solutions.
\end{abstract}

\maketitle

\section{Introduction}
Langmuir--Blodgett films \cite{B1935,L1917} are obtained by deposition of a thin film of amphiphilic molecules on a solid substrate, through a mechanism known as the Langmuir--Blodgett transfer.
In the typical setting, the solid substrate is immersed in a trough, where the amphiphiles float freely in the disordered state at the surface of a liquid.
The substrate is lifted vertically with velocity $\mathbf{v}$ and the amphiphiles, whose concentration on the surface of the liquid is kept constant, depose on it.
This process has very many industrial applications \cite{O92,Roberts}, and also lends itself to theoretical investigation, in particular for some aspects related to pattern formation. 
Indeed, according to the extraction velocity $\mathbf{v}$, the amphiphiles on the substrate can remain in the disordered (liquid-expanded) phase, align in the ordered (liquid-condensed) phase, or even form striped or checkerboard-like patterns in which the two phases alternate regularly \cite{bottom-up,confinement,science}.

To model that the liquid expanded and the liquid-condensed phases are preferred ones, one usually resorts to a potential which attains its minimum value precisely at those configurations.
From the mathematical viewpoint, the paradigmatic tool to describe this and similar phenomena is the celebrated Cahn--Hilliard equation, originally proposed to study phase separation phenomena \cite{CH1958}, see \cite{NC08} and the references therein, and also \cite{Miranville} for a recent review of this subject.
A few years ago, a variation of the classical Cahn--Hilliard equation was proposed in \cite{KGFT12,wilczek}  to account for the extraction velocity $\mathbf{v}$; moreover, the typical double-well potential $W_0(s)=\frac14(s^2-1)^2$ was also modified to allow modeling a \emph{meniscus}, the phenomenon generated by the asymmetry between adhesion and cohesion forces between a liquid and a solid.

By including these elements in the Cahn--Hilliard equation, the following model for the Langmuir--Blodgett films in a square domain of side length $L>0$ was presented in \cite{wilczek}
\begin{equation}\label{eq1}
\begin{cases}
c_t = \div \big(\nabla (-\Delta c-c+c^3+\nu \zeta(\bx))-\bv c\big) & \text{for $(\bx,t)\equiv(x,y,t)\in(0,L)^2\times(0,+\infty)$,}\\
c(0,y,t) = c_0 & \text{for $(y,t)\in [0,L]\times(0,+\infty)$,}\\
c_x(L,y,t)=c_{xx}(0,y,t)=c_{xx}(L,y,t)=0 & \text{for $(y,t)\in[0,L]\times(0,+\infty)$,}\\
c(x,0,t)=c(x,L,t) & \text{for $(x,t)\in[0,L]\times(0,+\infty)$,}\\
\end{cases}
\end{equation}
where $c\colon(0,L)^2\times(0,+\infty)\to\mathbb{R}$ describes the time-dependent concentration of the amphiphiles ($c=-1$ corresponding to the disordered phase and $c=1$ to the ordered one), $\zeta\colon(0,L)^2\to(-1,0)$ describes the profile of the meniscus, and $\mathbf{v}$ is the extraction velocity. 
In \cite{wilczek}, the explicit choice 
$$\zeta(\mathbf{x})= -\frac{1}{2}\bigg[1+\tanh\bigg(\frac{x-x_{\text{mns}}}{l_{\text{mns}}}\bigg)\bigg]$$
was made, modeling a meniscus at $x=x_{\text{mns}}$ of thickness $l_{\text{mns}}$\,; despite this specific example, it will always required to be analytic.
The vertical axis is denoted by~$x$, the horizontal one by~$y$, and with this convention, the extraction velocity $\mathbf{v}=(\beta,0)$ ($\beta>0$) points upward; the first boundary condition models the constant concentration at the liquid free surface; the second boundary conditions are designed to model the fact that the boundary at $x=L$ has no impact on the outflow of the concentration; finally, the last boundary condition encodes periodicity in the horizontal direction.

By denoting the potential by 
\begin{equation}\label{df-W}
W(\bx,s)=W_0(s) +\nu \zeta(\bx) s,
\qquad\text{where}\qquad
W_0(s)\coloneqq \frac14(s^2-1)^2,
\end{equation}
the first equation in \eqref{eq1} can be written as
$$c_t=\div\big(\nabla(-\Delta c+\partial_s W(\bx,c))-\bv c\big)\eqqcolon \div\big(\nabla \mu-\bv c\big),$$
where we have defined the \emph{chemical potential} $\mu\coloneqq -\Delta c+\partial_s W(\bx,c)$.

The one-dimensional version of \eqref{eq1} was studied in
\cite{BonDavMor2018}, where the change of variables $u(x,t)=c(x,t)-c_0$ was performed, yielding, see \cite[equation (2.2)]{BonDavMor2018},
\begin{equation}\label{eq2}
\begin{cases}
u_t = \big(-u_{xx}+(u+c_0)-(u+c_0)^3+\zeta(x)\big)_{xx}-\beta u_x & \text{for $(x,t)\in(0,L)\times(0,T)$,}\\
u(0,t)=u_x(L,t)=\mu(0,t)=\mu_x(L,t)=0 & \text{for $t\in(0,T]$,}\\
u(x,0)=u_0(x) & \text{for $x\in(0,L)$;}
\end{cases}
\end{equation}
the coefficient $\nu$ is taken to be equal to $1$, and the chemical potential $\mu$ reads, in this case,
\begin{equation}\label{df-mu}
\mu=\mu(u) =  -u_{xx} + (u+c_0)^3 - (u+c_0) + \zeta.    
\end{equation}

In \cite[Theorem~2.4]{BonDavMor2018}, the existence, uniqueness of weak solutions to \eqref{eq2}, and their regularity was studied.
Moreover, the authors established the existence of a global attractor when $\beta < L^{-3}$ (see \cite[Theorem~6.5]{BonDavMor2018}).  
However, a number of interesting questions linger:
\begin{itemize}
\item[(i)] Do solutions to \eqref{eq1} stabilize? In other words, do solutions approximate, for large times, a steady state solution?
\item[(ii)] Is it possible to extend the results of \cite{BonDavMor2018} to the original two-dimensional model of \cite{wilczek}? Is it possible to prove their stabilization? This latter question, which is possibly the most relevant for the physical application to the LB transfer, requires a thorough analysis of the set of steady states of \eqref{eq1},  which is not available at present or a new idea.
\end{itemize}
Question (i) above is particularly interesting when one can leverage the results obtained for $\beta=0$ (situation in which problem~\eqref{eq2} is a Cahn--Hilliard equation with potential~$W$ given by~\eqref{df-W}) to the case $\beta>0$, for which~\eqref{eq2} features the non-trivial advective term $-\beta u_x$\,.

In the present paper, we will concentrate on giving an affirmative answer to the first question. Our main result is the following.
\begin{theorem}\label{thm_main}
There exist a positive number $\Lambda_0$ and an
at most countable set $E\subset [\Lambda_0,\infty)$ such that, if $L\in (0,\infty)\setminus E$, then there is $\beta^*=\beta^*(L)>0$
with the following properties. If   $u_0\in \{u\in H^1(0, L)\colon u(0)=0\}$ and $u$ is the corresponding solution to \eqref{eq2} with $\beta\in[0, \beta^*]$, then there exists $u^\infty\in C^\infty(0,L)$ such that for all $m\in \bN$ we have 
\begin{equation}\label{eq3}
\lim_{t\to+\infty} \| u(t) - u^\infty  \|_{H^m(0,L)} =0.
\end{equation}
\end{theorem}
We stress that the notion of solution has to be clarified and we do this later (see Definition~\ref{df-weak} and formula \eqref{r-cont}). The set $E$, which is defined in~\eqref{def_E} below, of exceptional sample lengths lacks explicit characterization. 
The origin of its presence {comes from the definition of \emph{branch points} (see Definition~\ref{def:branchpoint}), which are those for which the Implicit Function Theorem cannot be applied.
Fortunately, Theorem~\ref{cor:discreteL} ensures that $E$ is discrete and therefore negligible.

In order to prove Theorem~\ref{thm_main} we will show that, for sufficiently small $\beta>0$, equation \eqref{eq2}, which is a perturbation of a gradient flow of the energy functional
$$
\mathcal{J}(u) = \int_0^L \bigg(\frac12 (u_x)^2 + W(x,u+c_0)\bigg)\,\de x,
$$
is still a gradient flow, albeit, in a weaker sense, see \cite{CLR} or \cite{RW23}. 
This weaker notion is based on the structure of the global attractor and the lack of homoclinic structures. 
Once we show this, it suffices to check that there is only a finite number of steady state solutions to \eqref{eq2}. 
The strategy of the proof is borrowed by that in \cite{RW23}, where it was applied to study the stabilization of a sixth-order Cahn--Hilliard-type equation featuring a convective term.

We start our work with the  analysis of steady states of \eqref{eq2} when $\beta=0$. 
We first estimate the number of the steady states for any $L>0$ by means of the shooting method, see Theorem \ref{t201}. 
We separately show that there exists $L_0>0$ such that for all $L\in (0, L_0)$ there is a unique steady state of \eqref{eq2}, see Proposition \ref{p-uni}.

Then we prove that if $v^1_0, \ldots, v_0^{N(L)}$ is the family of all steady states of \eqref{eq2} with $\beta=0$, then each $v_0^i$ belongs to a family $v_\beta^i$ for small $\beta\ge 0$. 
A natural tool for proving this is the Implicit Function Theorem, see \cite{NirenbergBook}, which requires triviality of the kernel of the linearization of the right-hand side of \eqref{eq2} with $\beta=0$. 
This is the main technical challenge in the study of the problem with $\beta=0$, which we resolve in Propositions \ref{prop:IFT-branch} and \ref{lem:equivalence} and in Theorem \ref{cor:discreteL}. 
Remarkably, in the course of its proof, we invoke the Weierstrass Preparation Theorem, see \cite{loja}.

The second step is to invoke \cite[Theorem 5.26]{CLR} to conclude that \eqref{eq2} is still a gradient flow for small $\beta$. 
Since this theorem is written in the language of the strongly continuous semigroups, we rewrite the existence result of \cite{BonDavMor2018} in this language, which, in particular, proves to be very useful when we have to establish the asymptotic collective compactness of families of semigroups $\{ S_\beta\}_{\beta\geq 0}$ generated by \eqref{eq2}. 
On this occasion, we prove some estimates on a modified energy functional (see \eqref{E1}) by using some ideas from \cite{eden} that were also employed in a similar context in \cite{KNR2016}.
This yields estimates that are uniform in time and in $\beta$ on the weak solutions to \eqref{eq2}.

The last step to prove stabilization is the application of \cite[Theorem~5.26]{CLR}, which we state in Theorem~\ref{thmRob} below for the reader's convenience. 
This theorem provides a check-list of four conditions to be satisfied to obtain the stability of gradient semigroups. We devote Section~\ref{stabilization} to showing that our problem satisfies these conditions.

We notice that the proof strategy that we use is intrinsically one-dimensional and cannot be easily established for the original two-dimensional problem~\eqref{eq1}. 
In particular, the question of counting the number of steady states of the two-dimensional elliptic problem~\eqref{eq1} is much more difficult, see, \emph{e.g.}, \cite{alikakos}. 
Because of these reasons, the analysis of~\eqref{eq1} is left for future work. 

\smallskip

The plan of the paper is the following.
In Section~\ref{betazero}, we study the steady states in the non advective case, namely problem~\eqref{eq2} for $\beta=0$ (see \eqref{sys_beta=0} below). 
The main result for this case is Theorem~\ref{t201}, where we prove that \eqref{r-stat} 
has a finite number of solutions.
The brief Section~\ref{smallbeta} deals with the small-$\beta$ continuation: by continuity (in fact analyticity) the results of Section~\ref{betazero} can be extended to~\eqref{eq2} with $\beta\in(0,\beta_0)$, where the existence of $\beta_0>0$ is proved in Theorem~\ref{thm:IFT}.
In Section~\ref{semigroupBG}, we translate the results in the language of analytic semigroup theory, to prepare the ground for the application of \cite[Theorem~5.26]{CLR}.
Finally, in Section~\ref{stabilization}, we prove stabilization of solutions.

\section{The case $\beta = 0$}\label{betazero}
We begin by studying problem~\eqref{eq2} in the regime $\beta=0$, namely
\begin{equation}\label{sys_beta=0}
    \begin{cases}
    u_t = \big( -u_{xx} + (u+c_0)^3 - (u+c_0) + \zeta \big)_{xx} & \text{for $(x,t)\in(0,L)\times(0,T)$,}\\[2pt]
    u(0,t) = u_{x}(L,t) = 0 & \text{for $t\in(0,T)$,}\\[2pt]
    \mu(0,t) = \mu_x (L,t) = 0 & \text{for $t\in(0,T)$,}
    \end{cases}
\end{equation}
where $\mu$ is defined by \eqref{df-mu}.
We have to address different aspects of   
solutions to the associated steady-state problem, that is
\begin{equation}\label{r-stat}
 \begin{cases}
    0= \big( -u_{xx} + (u+c_0)^3 - (u+c_0) + \zeta \big)_{xx} & \text{for $x\in(0,L)$,}\\[2pt]
    u(0) = u_{x}(L) = 0, \\[2pt]
    \mu(0) = \mu_x (L) = 0.
    \end{cases}
\end{equation}
We start with a preliminary lemma.
\begin{lemma}\label{lemma_0}
A function $u$ is a solution to~\eqref{r-stat} if and only if it is a solution to
\begin{equation}\label{r-red}
 \begin{cases}
    0= -u_{xx} + (u+c_0)^3 - (u+c_0) + \zeta  \equiv \mu & \text{for $x\in(0,L)$,}\\[2pt]
    u(0) = u_{x}(L) = 0 .
    \end{cases}
\end{equation}
Namely, $u$ is steady state if and only if the chemical potential vanishes, $\mu=0$.
\end{lemma}
\begin{proof}
If $u$ solves~\eqref{r-red}, then it is immediate to see that it solves~\eqref{r-stat}. 
Viceversa, integration of the equation in~\eqref{r-stat} twice over $(0,x)$ yields that the chemical potential is an affine function, that is 
$\mu(x) = a x +b$, for $a,b\in\mathbb{R}$.
Now, the boundary conditions \eqref{r-stat}$_3$ imply that $a= b=0.$
\end{proof}

We  have to address various aspects of solutions to \eqref{r-red}. In particular, we need estimates on the derivative $u_x$\,.  
We have to bound the number of steady states and we have to expose their dependence on the length $L$ of the domain.  
We prove estimates on $u_x(0)$ which are uniform in~$L$ and will be important for further analysis.

\begin{lemma}\label{LMnew}
If $u\in C^2([0,L])$ is a solution of \eqref{r-red}, then
\[
\sqrt2 z_\ell \coloneqq-|c_0^2-1|
\le \sqrt2\,u_x(0) \le \sqrt{5-4c_0 +(c_0^2-1)^2} 
\eqqcolon \sqrt2 z_r
\,.
\]
\end{lemma}
\begin{proof}
Multiplying \eqref{r-red} by $u_x$ and integrating from $0$ to $x\in[0,L]$ yields the first integral
\begin{align}
2u_x^2(x)
&= 2u_x^2(0)
    + (u(x)+c_0)^4 - 2(u(x)+c_0)^2 
    - c_0^4 + 2c_0^2
    + 4\int_0^x u_x(s)\zeta(s)\,\de s \nonumber\\
&=
 2u_x^2(0)
    + \big((u(x)+c_0)^2 - 1\big)^2
    - (c_0^2-1)^2
    + 4\int_0^x u_x(s)\zeta(s)\,\de s .
\label{eq:first_integral_corrected}
\end{align}
Let
\[
x_* \coloneqq \inf\{x\in(0,L]:\,u_x(x)=0\}.
\]
This is well-defined since $u_x$ is continuous and $u_x(L)=0$.

\medskip\noindent
\emph{Lower bound.}
If $u_x(0)\ge 0$, then $\sqrt2\,u_x(0)\ge 0\ge -|c_0^2-1|$ and we are done.
Assume $u_x(0)<0$. Then by definition of $x_*$ we have $u_x<0$ on $[0,x_*)$, hence
$u_x(s)\zeta(s)\ge 0$ for every $s\in(0,x_*)$ because $\zeta\le 0$.
Evaluating \eqref{eq:first_integral_corrected} at $x=x_*$ and using $u_x(x_*)=0$ gives
\[
0=2u_x^2(x_*) \ge 2u_x^2(0) - (c_0^2-1)^2,
\]
so $2u_x^2(0)\le (c_0^2-1)^2$, and since $u_x(0)<0$,
\[
\sqrt2\,u_x(0)\ge -|c_0^2-1|.
\]

\medskip\noindent
\emph{Upper bound.}
If $u_x(0)\le 0$, then $\sqrt2\,u_x(0)\le 0\le \sqrt{5-4c_0+(c_0^2-1)^2}$ and we are done.
Assume $u_x(0)>0$. Then $u_x>0$ on $[0,x_*)$. Using $-1\le \zeta\le 0$ we infer, for any
$x\in(0,x_*]$,
\[
4\int_0^x u_x(s)\zeta(s)\,\de s \;\ge\; -4\int_0^x u_x(s)\,\de s \;=\; -4u(x),
\]
where we used $u(0)=0$ in the last identity.
Insert this estimate into \eqref{eq:first_integral_corrected} and evaluate at $x=x_*$ to obtain
\begin{equation}\label{eq:upper_aux}
0=2u_x^2(x_*)
\ge 2u_x^2(0) + \big((u(x_*)+c_0)^2-1\big)^2 - (c_0^2-1)^2 -4u(x_*).
\end{equation}
Set $t \coloneqq u(x_*)+c_0$ (so $u(x_*)=t-c_0$). Then the sum of the second and last term in the right-hand side of \eqref{eq:upper_aux} become
\[
(t^2-1)^2 -4(t-c_0) = \big((t^2-1)^2-4t+5\big) + 4c_0 -5.
\]
The key point is the global algebraic inequality (valid for all $t\in\R$)
\begin{equation}\label{eq:sos_poly}
(t^2-1)^2-4t+5
= t^4-2t^2-4t+6
= (t^2-2)^2 + 2(t-1)^2
\;\ge\; 0.
\end{equation}
Using \eqref{eq:sos_poly} in \eqref{eq:upper_aux} therefore yields
$0\ge 2u_x^2(0) + (4c_0-5) - (c_0^2-1)^2$, \emph{i.e.},
\[
2u_x^2(0)\le (c_0^2-1)^2 -4c_0 + 5.
\]
Since $u_x(0)>0$, taking square roots gives
\[
\sqrt2\,u_x(0)\le \sqrt{5-4c_0+(c_0^2-1)^2},
\]
which is the desired upper bound.
\end{proof}

In order to expose the dependence of solutions to \eqref{r-red} on $L$ we will rewrite this problem in a fixed domain. For this purpose we 
set \(x=Ly\) with \(y\in[0,1]\) and \(\tilde v(y) \coloneqq u(Ly)+c_0\). 
The steady problem \eqref{r-red} becomes
\begin{equation}\label{eq:adim-BVP}
\begin{cases}
 -\tilde v_{yy} + L^2\big( \tilde v^3 - \tilde v + \tilde\zeta(y,L)\big) = 0, & \text{for $y\in(0,1)$,} \\[2pt]
 \tilde v(0)=c_0,\qquad \tilde v_y(1)=0,
\end{cases}
\end{equation}
where \(\tilde\zeta(y,L) \coloneqq \zeta(Ly)\). Of course, problems \eqref{r-red} and \eqref{eq:adim-BVP} are equivalent.

We will use  the shooting method.  For this purpose we set up the initial value problem for the first equation in \eqref{eq:adim-BVP},
\begin{equation}\label{eq:adim-IVP}
\begin{cases}
 -\tilde v_{yy}(y) + L^2\big(\tilde v(y)^3 - \tilde v(y) + \tilde\zeta(y,L)\big) = 0, & \text{for $y\in (0,1)$,}\\[2pt]
 \tilde v(0)=c_0,\qquad \tilde v_y(0)=\tilde z.
\end{cases}
\end{equation}

In order to proceed, we have to make sure that solutions to \eqref{eq:adim-IVP} exist on $[0,1]$. This is the content of the lemma below.

\begin{lemma}\label{lem:persistence}
For any value of $c_0$ and $\tilde z\in [Lz_\ell\,, L z_r]$ there exists a unique solution $\tilde v(\cdot;L,\tilde z)$ to  \eqref{eq:adim-IVP} on $[0,1]$. 
Moreover, if $\tilde v$ has a maximum in $[0,1)$ or $\tilde v_y(1) =0$ when $\tilde v_{y}(0)>0$, (respectively $\tilde v$ has a minimum in $[0,1)$ or $\tilde v_y(1) =0$ when $\tilde v_{y}(0)<0$), then we have the following bounds,
\[
\sup_{y\in[0,1]}|\tilde v(y;L,\tilde z)|\le \max\{ 1, v_M\},\qquad
\sup_{y\in[0,1]}|\partial_y\tilde v(y;L,\tilde z)|\le 
L\max\{ |z_\ell|, z_r\} + L^2 \max\{ 1, v_M\}.
\]
where $v_M>0$ is the (only real) solution to $v_M^3 - v_M = 1$.
\end{lemma}
\begin{proof}
We will consider only the case $\tilde v_y(0)>0$, since the case $\tilde v_y(0)<0$ can be handled in a similar way and $\tilde v_y(0)=0$ is a limit of both cases and does not bring any additional difficulty.

Set $y_0=0$ and suppose $\tilde v_y(0)>0$; then the function $\tilde v$ is increasing on an interval $[0, y_1]$ and it attains its maximum at $y_1\in (0,1]$\,. 
We claim that $\tilde v_{yy}(y_1)\le 0$, to show which we distinguish two cases: $y_1=1$ and $y_1<1$. 
In the former case, we have that $\tilde v_y\ge0$ on $[0,1]$, hence the condition $\tilde v_y(1)=0$ implies that $\tilde v_y$ attains at $y=1$ its minimum and our claim follows.
In the latter case, we deduce that $\tilde v_{yy}(y_1)\le 0$ and equation \eqref{eq:adim-IVP} implies 
$$
\tilde v(y_1)^3 - \tilde v(y_1) + \tilde\zeta(y_1,L) \le 0.
$$
Since $-1 < c_0< \tilde v(y_1)$ we deduce that $\tilde v(y_1) \le v_M,$ where $v_M$ is a solution to 
$$
v^3_M - v_M = 1.
$$
If $y_1 =1$ our argument is finished, whereas,
if $y_1<1$, we may suppose that $\tilde v$ is decreasing on $[y_1, y_2]$ and it achieves a local minimum at $y_2\in(y_1,1]$. 
If $y_2<1$, then obviously  $\tilde v_{yy}(y_2)\ge 0$. If $y_2=1$, we notice that $\tilde v_y\le 0$ on $[y_1, y_2]$, hence $\tilde v_y$ attains a maximum at $y =1$, hence  $\tilde v_{yy}(x_2)\ge 0$.

Then, equation  \eqref{eq:adim-IVP} yields 
$$
\tilde v(y_2)^3 - \tilde v(y_2) + \tilde\zeta(y_2,L) \ge 0.
$$
Since $\tilde \zeta<0$, we deduce that $ \tilde v$ must be greater than or equal to $v_m<0$, the solution to 
$v^3_m - v_m = 0.
$
Hence $v_m = -1.$

Thus, we see that $\tilde v$ varies on intervals $\{[y_{k-1}, y_{k}]\}_{k=1}^N$ between $-1$ and $v_M$, where $y_k$  and $y_{k+1}$ are consecutive extrema. If $y_N<1$, then we can continue $\tilde v$ past $y_N$ to reach the next extremum. Thus, we conclude that $\tilde v(y) \in [-1, v_M]$ for all $y\in [0, 1]$. 

The bound on $\tilde v_y$ follows from the integration of equation \eqref{eq:adim-IVP}.
\end{proof}

We now define the function $f \colon (0,+\infty) \times \bR \to \bR$ by
\begin{equation}\label{eq:BC-function}
  f(L,\tilde z) \coloneqq \tilde v_y(1;L,\tilde z).
\end{equation}
\begin{remark}\label{rem:analyticity}
    The function $f$ defined in \eqref{eq:BC-function} is analytic in $L>0$ and $\tilde z$.
    This is a consequence of the analyticity of $\zeta$ and of the fact that the nonlinearity in \eqref{eq:adim-IVP} is polynomial.
    Indeed, we can assume that $L,\tilde z\in \bC$ and therefore standard differentiability with respect to parameters yields analyticity of $f$ with respect to $L$ and $\tilde z$. Restricting  $L$ and $\tilde z$ to the real line completes the argument.
\end{remark}
Here is the main observation of this section.

\begin{theorem}\label{t201}
For any given $L>0$ and $c_0$\,, there is a finite number of solutions to \eqref{r-red}. 
They correspond to the zeros of the analytic function $\tilde z\mapsto f(L,\tilde z)$, where $f$ is defined in \eqref{eq:BC-function}.
\end{theorem}
\begin{proof}
Solutions to \eqref{eq:adim-IVP} depend analytically on three parameters $c_0$\,, $L$, and $\tilde z$ (see Remark~\ref{rem:analyticity}). 
Solution to \eqref{r-red} correspond to zeros of the function $\tilde z\mapsto f(L,\tilde z)$. 
From Lemma~\ref{LMnew}, we know that these zeros may belong only to the interval $[Lz_\ell, Lz_r]$, which is compact.
Therefore, analyticity with respect to $\tilde z$ implies that there is only a finite number of them.
\end{proof}

We may be more precise when  the sample length, $L$, small. 
In this case we may prove uniqueness of  steady states.
\begin{proposition}\label{p-uni}
There is $\Lambda_0>0$ such that if $L\in (0, \Lambda_0)$, then there exists a unique solution to \eqref{r-red}.
\end{proposition}
\begin{proof}
Since problems \eqref{r-red} and \eqref{eq:adim-BVP} are equivalent, we deal with \eqref{eq:adim-BVP} which is more manageable.

Let us notice that when $L=0$, then  \eqref{eq:adim-BVP} has a unique solution given by $\tilde{v}(y) = c_0$\,. We shall consider  \eqref{eq:adim-BVP} as a problem 
\begin{equation}\label{r-F}
G(L, \tilde{v}(L)) =0,
\end{equation}
where 
$$
G\colon \bR \times C^2_{\text{bc}}([0,1])\to C([0,1])
$$
and $C^2_{\text{bc}}([0,1])=\{\tilde{v}\in C^2([0,1]): \tilde{v}(0)=c_0,\ \tilde{v}_y(1) =0\}$. 
In order to apply the Implicit Function Theorem, we have to check that $\partial_{\tilde{v}} G(0,c_0)$ is an isomorphism. Indeed, 
$$\frac{\partial G}{\partial \tilde{v}}(0,c_0)h = h_{yy}\,.$$
This operator with the boundary conditions $ h(0)=h_y(1) =0$ has a trivial kernel and it is an isomorphism. Hence, there is $\Lambda_0>0$ such that for all $L\in (-\Lambda_0, \Lambda_0)$ there is a function $(-\Lambda_0, \Lambda_0)\ni L \mapsto \tilde{v}(L) $ is a unique solution to \eqref{r-F}.
\end{proof}

In view of Theorem~\ref{t201}, for each $L>0$, we denote by $\tilde z^1(L),\ldots,\tilde z^{N(L)}(L)$ the zeros of $\tilde z\mapsto f(L,\tilde z)$ and by $\tilde v^1(L),\ldots,\tilde v^{N(L)}(L)$ the corresponding solutions to~\eqref{r-red}.

We now analyze the set \(\{\tilde z^j(L)\}_{j=1}^{N(L)}\) in more detail.
In particular, we partition the pairs $(L,\tilde z^j(L))$ in two classes. 
\begin{definition}[Nondegenerate solution and branch point]\label{def:branchpoint}
Let $L_0>0$ and consider a zero $\tilde z_0 \in \big\{ \tilde z^j(L_0) \big\}_{j=1}^{N(L_0)}$.
\begin{itemize}
\item We call $(L_0,\tilde z_0)$ \emph{nondegenerate} if $\partial_{\tilde z}f(L_0,\tilde z_0)\neq 0$.
\item We call $(L_0,\tilde z_0)$ a \emph{branch point} if $\partial_{\tilde z}f(L_0,\tilde z_0)=0$.
\end{itemize}
\end{definition}

Our aim is to relate nondegeneracy of pairs $(L_0,\tilde z_0)$ to the triviality of the kernel of the linearized operator corresponding to problem~\eqref{r-stat}.
This will be developed in the following results.

We start by considering the linearized problem associated with~\eqref{eq:adim-BVP}.
For a given $L_0>0$, let~$\tilde v_*$ be a solution to \eqref{eq:adim-BVP} for $L=L_0$\,. We set $D(\cM_{L_0}) = \{ u\in H^2(0,1)\colon u(0) = 0 = u_x(1)\}$ and define the linearized operator $\cM_{L_0} \colon D(\cM_{L_0})\subset L^2(0,1) \to L^2(0,1)$ as 
\begin{equation}\label{eq_lin_op}
\cM_{L_0}h \coloneqq -h_{yy} + L_0^2(3\tilde v_*^2(y)-1)h.
\end{equation}
\begin{proposition}[Nondegenerate points yield unique analytic branches]\label{prop:IFT-branch}
Let $L_0>0$ and let $\tilde z_0 \in \big\{ \tilde z^j(L_0) \big\}_{j=1}^{N(L_0)}$.
If $(L_0,\tilde z_0)$ is non degenerate, then there exist $\varepsilon>0$ and a unique real-analytic function
\[
\tilde z \colon (L_0-\varepsilon,L_0+\varepsilon)\to \mathbb R
\]
such that $f(L,\tilde z(L))=0$ for $|L-L_0|<\varepsilon$ and $\tilde z(L_0)=\tilde z_0$\,.
Consequently $L\mapsto \tilde v(\cdot;L,\tilde z(L))\in C^1([0,1])$ is a unique real-analytic branch of steady states near $L_0$\,.
\end{proposition}
\begin{proof}
The real-analytic Implicit Function Theorem (see, \emph{e.g.}, \cite[Section~6.1]{KP_book_IFT}) applies to the analytic map $f\colon (0,\infty)\times \mathbb{R}\to\mathbb{R}$ defined in \eqref{eq:BC-function}, at $(L_0,\tilde z_0)$, since $\partial_{\tilde z} f\neq 0$.
Analytic dependence of $L\mapsto \tilde v(\cdot;L,\tilde z(L))$ follows from Remark~\ref{rem:analyticity}.
\end{proof}

In the next proposition, we prove that the loss of nondegeneracy of pairs $(L_0,\tilde z_0)$ corresponds to the nontriviality of the kernel of the operator $\cM_{L_0}$\,.
\begin{proposition}\label{lem:equivalence}
Let $(L_0,\tilde z_0)$ satisfy $f(L_0,\tilde z_0)=0$ and let $\tilde v_*(y) \coloneqq \tilde v(y;L_0,\tilde z_0)$ denote the corresponding steady state, \emph{i.e.}, $\tilde  v_*$ is a solution to~\eqref{eq:adim-BVP}.
Consider the boundary-value problem 
\begin{equation}\label{eq_lin_BVP}
    \begin{cases}
    \cM_{L_0} h=0 & \text{for $y\in(0,1)$},\\[2pt]
    h(0)=0,\quad h_y(1)=0,
    \end{cases}
\end{equation}
for the linearized operator introduced in \eqref{eq_lin_op}.
Then $\partial_{\tilde z} f(L_0,\tilde z_0)=0$ if and only if $\ker \cM_{L_0}\neq \{0\}$.
\end{proposition}

\begin{proof}
Differentiate the initial-value problem~\eqref{eq:adim-IVP} with respect to $\tilde z$ at $(L_0,\tilde z_0)$.
The variation $h(y)=\partial_{\tilde z}\tilde v(y;L_0,\tilde z_0)$ solves
\begin{equation}\label{eq:lin-adim}
 \begin{cases}
    \cM_{L_0} h=0 & \text{for $y\in(0,1)$,}\\[2pt]
    h(0)=0,\quad h_y(0)=1.
    \end{cases}
\end{equation}
Hence $\partial_{\tilde z}f(L_0,\tilde z_0)=h_y(1)$. 
Thus, $\partial_{\tilde z}f(L_0,\tilde z_0)=0$ if and only if the solution to \eqref{eq:lin-adim} is also a solution to \eqref{eq_lin_BVP}, \emph{i.e.}, if and only if \eqref{eq_lin_BVP} admits a nontrivial solution. 
Conversely, if $k$ is a nontrivial solution to \eqref{eq_lin_BVP}, then $k_y(0)\neq 0$ by uniqueness; defining $h \coloneqq k/k_y(0)$ gives $h(0)=0$, $h'(0)=1$, $h'(1)=0$, hence $\partial_{\tilde z}f(L_0,\tilde z_0)=0$.
\end{proof}
Therefore, branch points correspond to nontriviality of the kernel of the linearized operator $\cL_{L_0}$\,. 
In the next theorem, we prove that there are at most countably many branch points.
\begin{theorem}[Rarity of branch points in $L$] \label{cor:discreteL}
The set 
\begin{equation}\label{def_E}
E \coloneqq \big\{L>0: \exists\, j\in\{1,\ldots,N(L)\} : \text{$(L,\tilde z^j(L))$ is a branch point}\, \big\} 
\end{equation}
is discrete (in particular, of measure zero).
\end{theorem}
\begin{proof}
Define the zero level sets, for $i=0,1$,
$$
\cS_i \coloneqq \big\{(L,\tilde z)\in (0,\infty) \times [Lz_\ell \,, Lz_r] \ :\ \partial^i_{\tilde z} f(L,\tilde z)=0 \big\}. 
$$
We immediately notice  that the set $\cS_0\setminus \cS_{1}$ is a smooth manifold of dimension $1$ as a consequence of the Implicit Function Theorem. 
Consider any compact set $K\subset(0,+\infty)$ and the projection $\pi_L \colon (0,+\infty) \times \bR\to (0,+\infty)$ on the first component. 

We claim that $\pi_L \big(\cS_0\cap\cS_1\cap (K\times [Lz_\ell \,, Lz_r]) \big) \subset K$ is finite. 
If the set $\cS_0\cap\cS_1\cap (K\times [Lz_\ell \,, Lz_r])$ is finite, then the claim follows.
Otherwise, it has an accumulation point and two cases may occur.
\begin{itemize}
\item[(a)] There exists $L_0\in K$ such that $\cS_0\cap \cS_1\cap(\{L_0\}\times [Lz_\ell \,, Lz_r])$ has an accumulation point, hence the same is true about $\cS_0 \cap(\{L_0\}\times [Lz_\ell \,, Lz_r])$. 
As a result, the analytic function of one variable $\tilde z \mapsto f(L_0, \tilde z)$ vanishes on a set having an accumulation point. 
This means that $f(L_0, \tilde z) \equiv 0$, which implies that there are infinitely many steady states for $L=L_0$\,, contradicting Theorem~\ref{t201}.

\item[(b)] The set $\pi_L \big(\cS_0\cap \cS_1\cap (K \times [Lz_\ell \,, Lz_r]) \big)$ has an accumulation point. 
By invoking the Weierstrass Preparation Theorem (see, \emph{e.g.}, \cite[Chapter C, \S 2]{loja}), and possibly up to a local analytic change of coordinates, we can factor
\[
f(L,\tilde z) = U(L,\tilde z) P_L(\tilde z),
\]
where $U$ is an analytic nonvanishing function, and $P_L$ is a monic polynomial of degree $n$, whose coefficients are analytic in $L$. 
We recall that the discriminant of $P_L$ is a function of $L$ and is defined by
$$
d(L) = (-1)^{\binom{n}{2}} \prod_{j=1}^n \frac{\partial P_L}{\partial\tilde z}(\tilde z_j),
$$
where $\tilde z_j$\,, $j=1,\ldots, n$ is the complete set of zeros of the polynomial $\tilde z\mapsto P_L(\tilde z)$.  

In this case, $L\mapsto d(L)$ is an analytic function of one variable vanishing on a set with an accumulation point in a compact set (in the intersection $\cS_0\cap \cS_1$, zeros are multiple). 
Thus, $d$ must be identical to zero on $(0, \infty)$. 
However, we know that for small values of $L$ there is a unique and simple zero of $f(L,\cdot)$, which contradicts Proposition \ref{p-uni}.
\end{itemize}
The theorem is proved.
\end{proof}

We noted that problems \eqref{r-red} and \eqref{eq:adim-BVP} are equivalent, {\it i.e.}, $u$ is a solution to  \eqref{r-red} if and only if $\tilde v$,  the scaling of $u$, is a solution to \eqref{eq:adim-BVP}. 
Moreover, problem \eqref{eq_lin_BVP} has a nontrivial solution  if and only if there is a non-zero solution of
\begin{equation}\label{eq_linear}
    \begin{cases}
 -h_{xx}   + (3u^2(x)-1) h =0, & \text{for $x\in (0,L)$,}\\[2pt]
 h(0) = 0, \quad h_x(L) =0.
    \end{cases}
\end{equation}
We can summarize the results of this section as follows.
\begin{theorem}
Let us suppose $L\in(0,\infty)\setminus E,$ where $E$ is defined in \eqref{def_E} and $c_0\in(-1,0)$. Then there exists a finite family $\{u^1,\ldots, u^{N(L)}\}$ of steady stated of \eqref{r-stat} and the linearized operator at each $u^i$, {\it i.e.}, \eqref{eq_linear} has a trivial kernel.
\end{theorem}
\begin{proof}
We noticed in Lemma \ref{lemma_0} that equations \eqref{r-stat} and \eqref{r-red} are equivalent. 
Existence of a finite family of solutions to \eqref{r-red} for any positive $L$ and $c_0$ follows from Theorem \ref{t201} after taking into account that  \eqref{r-red}  and the rescaled equation \eqref{eq:adim-BVP} are equivalent. 
Proposition \ref{lem:equivalence} and Theorem \ref{cor:discreteL} combined show that if $L\in (0,\infty) \setminus E$, then all solutions to \eqref{r-red}  lead to linearized operators \eqref{eq_linear} with a trivial kernel.
\end{proof}

\section{Continuation to small \texorpdfstring{\(\beta\)}{beta}}\label{smallbeta}

We first address existence of steady state solution of \eqref{r-stat} for a range of $\beta.$ Later, in the course of proof of Lemma \ref{part-d}, we will study the spectrum of the linearized operator.
We return to the fourth-order formulation with \(\beta\) and we define the operator $\cF \colon X\times (-1,1)\to L^2(0,L)$ by formula
\[
\mathcal F(u,\beta) \coloneqq \big(-u_{xx} + (u+c_0)^3 - (u+c_0) + \zeta\big)_{xx} - \beta u_x\,,
\]
where $X$ is the following complete metric space
\begin{equation}\label{domainX}
X \coloneqq \{ u\in H^4(0,L)\colon u(0)=0,\ u_x(L)=0,\ \mu(0)=\mu_x(L)=0\}.
\end{equation}

At a steady state \(u_*\) for \(\beta=0\), the linearisation \(D_u\mathcal F(u_*,0)= \Delta \cM_L\) has kernel \(\{0\}\) provided the second-order linearised problem \eqref{eq:lin-adim} is nondegenerate, see Proposition~\ref{lem:equivalence}. 
Then, we can prove the following result.
\begin{theorem}[small-\(\beta\) continuation] \label{thm:IFT}
Suppose $0<L\notin E$, where $E$ is defined in~\eqref{def_E}, and let 
$u^j(L)$, for some $j\in\{1,\ldots,N(L)\}$, be a steady state of \eqref{sys_beta=0}.
Then the linearized operator at \(u^j(L)\) has trivial kernel and there exist \(\beta_0>0\) and a \(C^1\) function \((-\beta_0,\beta_0)\ni\beta\mapsto u_\beta^j \in H^4(0,L)\) with \(u_0^j=u^j(L)\) and \(\mathcal F(u_\beta^j,\beta)=0\). 
Moreover, there exists $C>0$ depending only on $\cF$ and $u^i_\beta$ such that
\begin{equation}\label{eq_small_beta_estimate}
\norm{u_\beta^j(L)-u^j(L)}_{H^4(0,L)} \le C\lvert \beta \rvert.
\end{equation}
\end{theorem}
\begin{proof}
The existence of the function $\beta \mapsto u_\beta^j$ is a direct application of the Implicit Function Theorem in Banach spaces (see, \emph{e.g.}, \cite[Theorem~2.7.2]{NirenbergBook}). Indeed,  the mapping \(\mathcal F\) is of class~\(C^1\) in a neighborhood of \((u^j(L),0)\) in $X$ defined in \eqref{domainX}. 
Moreover, due to  Proposition~\ref{lem:equivalence}, \(D_u\mathcal F(u^j(L),0)\) has trivial kernel by the nondegeneracy $(L,\tilde z^j(L))$.  Since  \(D_u\mathcal F(u^j(L),0)\) is the inverse of a compact operator, its spectrum consists of eigenvalues, hence  \(D_u\mathcal F(u^j(L),0)\) is an isomorphism.
\end{proof}

\section{Semigroup background}\label{semigroupBG}
In this section we will recall the existence result of \cite{BonDavMor2018}. We want to recast it in the language of the analytic semigroup theory as exposed in \cite{engel} or \cite{henry} (to which we redirect the reader for basic facts). For this purpose we introduce in a concise way the basics  of this  theory. We have to do this, 
because  our main tool, \emph{i.e.}, \cite[Theorem 5.26]{CLR}, is written in this language. In fact, this theorem provides a check-list of properties of the semigroup generated by an equation to be established. 

\subsection{Existence of weak solutions to \eqref{eq2}}
Here, we recall the statement of the main existence result in \cite{BonDavMor2018}. We first introduce a number of function spaces. We set
$$V \coloneqq \{u \in H^1 (0, L) : u(0) = 0\}$$
and we endow it with the scalar product $(u,v)_V\coloneqq \int_0^L u'(x)v'(x)\,\de x$.
We introduce the notion of weak solution to \eqref{eq3} as follows.
\begin{definition}
\label{df-weak}
Let $u_0 \in V$ be a given initial datum. A function $u$ is a weak solution to \eqref{eq2} in $[0, T ]$ corresponding to $u_0$ if $u(x, 0) = u_0 (x)$ for almost every $x \in (0, L)$, and the following conditions hold:
\begin{enumerate}
\item $u \in H^1 (0, T ; V' ) \cap L^2 (0, T ; H^3 (0, L));$
\item the chemical potential $\mu \coloneqq -u_{xx} + \partial_s W(x, u+c_0) \in L^2(0, L; V)$;
\item for every $\psi \in V$ and almost every  $t \in (0, T )$, the function $u$ satisfies
\begin{equation*}\label{r2.5}
\langle\partial_t u, \psi\rangle_{V, V'} = - (\mu, \psi)_V -\beta \int_0^L u_x(x)\psi(x) \, \de x;
 \end{equation*}
\item for almost every $t \in (0, T )$, the function $u$ satisfies the boundary conditions
$u(0, t) = u_x(L, t) = 0$.
\end{enumerate} 
\end{definition}
The following result was established in \cite{BonDavMor2018}.
\begin{proposition}(\cite[Theorem 2.4]{BonDavMor2018})\label{BDM_thm2.4}
Let $W$ be given by \eqref{df-W} and let $u_0 \in V$. For any $T > 0$, there exists a unique weak solution $u$ to the problem \eqref{eq2} in
$[0, T ]$ corresponding to $u_0$ in the sense of Definition \ref{df-weak}. Furthermore, setting
\begin{equation}\label{df-E}
\cE(v) \coloneqq \frac12 \int_0^L
 |v_x (x)|^2\, \de x + \int_0^L  W (x, v(x))\, \de x\qquad\hbox{for }v \in V,   
\end{equation}
the following energy equality holds true for almost every $t \in (0, T )$:
\begin{equation}\label{r-dE/dt}
\frac{\de}{\de t} \cE(u(t)) + \|\mu (t)\|^2_V = 
-\beta \int_0^{L}  u_x (x, t) \mu(x, t)\, \de x,
\end{equation}
where $\mu$ is defined in \eqref{df-mu}. Finally, the solution $u$ depends continuously on the
initial data, in the following sense: given any $M > 0$, there exists a constant
$C_M>0$ (depending on $\beta$, $L$, $M$, and~$T$) such that for every initial data $u_0\,,\bar u_0 \in V$ with $\|u_0\|_V\,, \|\bar u_0 \|_V \le  M$ the corresponding weak solutions $u$, $\bar u$
satisfy, for all $t \in [0, T ]$,
$$
\|u(t) - \bar u(t)\|_V \le C_M \|u_0 - \bar u_0 \|_V\,.
$$
\end{proposition}

\subsection{Basics of the analytic semigroup theory}\label{s-semi}
By considering a Hilbert space $\scrH$, we rewrite the problem in \eqref{eq2} as
\begin{equation}\label{eq_abstract_formulation}
w_t + A w = F_\beta(w), \qquad w(0)=w_0
\end{equation}
where $A \colon D(A)\subset \scrH \to \scrH$ is the (one-dimensional) bilaplacian operator $A=\Delta^2=(-\Delta)^2$, and where the meaning of $F_\beta$ will be explained momentarily. Here,  
$$
D(A) \coloneqq\{ w\in H^4(0,L): w(0) = w_{xx}(0) = 0, \ w_x(L)=w_{xxx}(L) = 0\}.
$$
It is not difficult to check that $A$ is a positive definite self-adjoint operator. Hence, it is sectorial on $\scrH=L^2(0,L)$ and $-A$ is a generator of an analytic semigroup. In particular, the fractional powers $A^\alpha$ are well-defined. 
We write $\scrH^\alpha \coloneqq D(A^\alpha)$; in particular, $A^{1/2} = - \Delta$ with
$$
\scrH^{1/2}= D(A^{1/2}) = \{ w\in H^2(0,L):\ w(0) = 0= w_x(L)\}
$$
and 
$$
\scrH^{1/4} = V, \qquad \| (-\Delta)^{1/2} w \| = \| w_x\|,
$$
where, here and hencefort, we use $\|\cdot\|$ to denote the $L^2(0,L)$-norm $\|\cdot\|_{L^2(0,L)}$\,. 
Indeed,
\begin{equation}\label{r-pol}
\| (-\Delta)^{1/2} w \|^2 = ( w, -\Delta w) = (w_x, w_x).
\end{equation}

We also recall existence of constants $C_\alpha$ for all $\alpha\in (0,1]$ such that, see \cite[Theorem 1.4.3]{henry},
\begin{equation}\label{r-spec} 
\|A^\alpha e^{-At} \| \le C_\alpha t^{-\alpha} e^{-\lambda t}, \qquad t>0,
\end{equation}
where $\lambda \in(0,\inf \sigma (A))$, $\sigma (A)$ being the spectrum of $A$.

In order to accommodate the boundary conditions coming from $\mu(0) = 0=\mu_x(L)$, we introduce $w = u - \eta$, where $\eta$ is a smooth function with the following properties,
$$
\eta(0)=0=\eta_x(L), \quad \eta_{xx}(0) = c_0^3 - c_0 +\zeta(0), \quad \eta_{xxx}(L) = \zeta_x(L).
$$
Then $F_\beta \colon \scrH^{1/2} \to \scrH$ given by
\begin{equation}\label{df-Fbe}
F_\beta(w) = \Delta \partial_s W(x, w + \eta + c_0) - \beta w_x - \Delta^2 \eta - \beta \eta_x
\end{equation}
is well-defined, and one can easily check that $F_\beta$ is locally Lipschitz continuous. 
As a result, the basic existence result \cite[Theorem 3.3.3]{henry} applies, yielding local-in-time strong solutions for data $w_0\in \scrH^{1/2}$, meaning that 
\begin{equation}\label{r-cont}
w \in C((0,T); D(A))\cap C([0,T); \scrH^{1/2})\cap C^1((0,T); \scrH).
\end{equation}
Let us note that the semigroup theory provides locally more regular solutions than the weak ones constructed in \cite{BonDavMor2018}. 

Since it suffices that the initial condition $u_0$ is in $V= \scrH^{1/4}$ to construct a weak solution, we cannot immediately claim that weak and strong solutions coincide. However, Proposition~\ref{BDM_thm2.4} assures us that for a.e.~$t_0>0$ the weak solution $u(t_0)$ is in $H^3(0,L)$ and satisfies the appropriate boundary conditions, hence it is in $\scrH^{1/2}$. Thus, the strong solution of \eqref{eq_abstract_formulation} with the nonlinearity \eqref{df-Fbe} and initial condition $w_0 = u(t_0) - \eta$ exists. By the uniqueness part of Proposition~\ref{BDM_thm2.4}, we see that the strong solution $w(t)+\eta$ and  the weak  solutions $u(t)$ must coincide on $[t_0, T)$. 
Since $t_0>0$ and $t_0<T$ are arbitrary, we see that both types of solutions agree on $(0,\infty)$.

We need to establish a number of properties of the solution.
The following integral representation of strong solutions will be very useful for us, see \cite{engel,henry},
\begin{equation}\label{r-convar}
\begin{split}
w(t) =&\, e^{-\Delta^2 (t-t_0)}w_0 + \int_{t_0}^t e^{-\Delta^2 (t-\sigma)}\big( \Delta \partial_s W(x, w + \eta + c_0) - \beta w_x - \Delta^2 \eta - \beta \eta_x \big)\, \de\sigma \\
=&\,  e^{-\Delta^2 (t-t_0)}w_0 + (Id - e^{- \Delta^2 (t-t_0)}) \eta \\
&\, + \int_{t_0}^t e^{-\Delta^2 (t-\sigma)}\big( \Delta \partial_s W(x, w + \eta + c_0) - \beta w_x - \beta \eta_x \big)\, \de\sigma. 
\end{split}
\end{equation}
Since we want to use it to deduce estimates of weak solutions, $t_0$ must be greater than zero.

Let us note the following smoothing effect.
\begin{lemma}[Instantaneous smoothing]\label{lem:instantaneous_smoothing}
Let $u_0\in V$ and let $u$ be the corresponding weak solution to \eqref{eq2}.
Then, for every $0<\tau<T<\infty$ and every $m\in \bN$, there exists a positive constant $C_{m,\tau,T}=C_{m,\tau,T}(u_0,L,\beta,\eta,\zeta)$
such that
\[
\sup_{t\in[\tau,T]}\|u(t)\|_{H^m(0,L)}\le C_{m,\tau,T}.
\]
In particular,
\[
u(t)\in C^\infty(0,L)\qquad\text{for every }t>0.
\]
\end{lemma}
\begin{proof}
Fix $0<\tau<T$. Since $u\in L^2(0,T;H^3(0,L))$ as per Definition~\ref{df-weak}(1), there exists
$t_0\in(0,\tau/2)$ such that $u(t_0)\in H^3(0,L)$ and satisfies the boundary
conditions.

Set $w=u-\eta$. By the discussion preceding \eqref{r-convar}, the
weak solution coincides on $[t_0,T]$ with the strong solution of
\eqref{eq_abstract_formulation} issued from $w(t_0)$.
In particular, $w\in C\bigl((t_0,T];D(A)\bigr)\cap C^1\bigl((t_0,T];\scrH\bigr)$, and, since $D(A)\hookrightarrow H^4(0,L)$, we infer that there exists a positive constant $C_{\tau,T}(u_0,L,\beta,\eta,\zeta)$ depending on $u_0$\,, $L$, $\beta$, $\eta$, and $\zeta$ such that
\[
\sup_{t\in[\tau,T]}\|w(t)\|_{H^4(0,L)}\le C_{\tau,T}(u_0,L,\beta,\eta,\zeta).
\]

We now assume that for some $m\ge 4$ there exists a constant $C_{m,\tau,T}(u_0,L,\beta,\eta,\zeta)>0$ such that
\[
\sup_{t\in[\tau,T]}\|w(t)\|_{H^m(0,L)}\le C_{m,\tau,T}(u_0,L,\beta,\eta,\zeta)
\]
and bootstrap.
Since $H^m(0,L)$ is a Banach algebra in one space dimension and $F_\beta$ is a
differential polynomial of order at most two, see \eqref{df-Fbe}, we have that there exists a positive constant $\widetilde{C}_{m,\tau,T}(u_0,L,\beta,\eta,\zeta)$ such that
\begin{equation}
\label{eq_banach_algebra_etc}
\sup_{t\in[\tau,T]}\|F_\beta(w(t))\|_{H^{m-2}(0,L)}\le \widetilde{C}_{m,\tau,T}(u_0,L,\beta,\eta,\zeta).
\end{equation}
Writing \eqref{r-convar} with initial time $t_0$ and applying
$A^{(m+1)/4}$, we obtain, thanks to \eqref{r-spec}, that
\begin{align*}
\|w(t)\|_{H^{m+1}(0,L)}
\le &\, C_{m,\tau,T}(u_0,L,\beta,\eta,\zeta)\|w(t_0)\|_{\scrH}
 + C_{3/4}\int_{t_0}^t (t-s)^{-3/4}
   \|F_\beta(w(s))\|_{H^{m-2}(0,L)}\,\de s \\
\le &\, C_{m,\tau,T}(u_0,L,\beta,\eta,\zeta)\|w(t_0)\|_{\scrH}
 + 4T^{1/4}C_{3/4}\widetilde{C}_{m,\tau,T}(u_0,L,\beta,\eta,\zeta) \\
 \eqqcolon&\, C_{m+1,\tau,T}(u_0,L,\beta,\eta,\zeta),
\end{align*}
for every $t\in[\tau,T]$.

Hence $w$ is bounded in $H^m(0,L)$ on $[\tau,T]$ for every
$m\in \bN$. Since $\eta$ is smooth, the same is true for $u$, possibly redefining the constant $C_{m,\tau,T}$\,.
\end{proof}

\subsection{Estimates uniform in $\beta$}
Our convergence analysis depends on uniform-in-time estimates on solutions in various Sobolev norms. On the way we will revisit the proof of existence a global attractor for \eqref{eq2} given in \cite[Theorem 6.5]{BonDavMor2018}. We will improve it by making it  independent of $L$ and $\beta\in (0,1]$. 
The first step in this direction in a new proof of existence of an absorbing ball. 
We will use, for this purpose, an argument based on the ideas of \cite{eden}, which were developed in \cite{KNR2016}. 
Then, we show uniform-in-time estimates in the Sobolev norms $H^m$ for all  $m\ge 2$. 
This is more than enough to establish the existence of a global compact attractor  $\cA_\beta$ in the $H^1$-topology. However, the uniform bounds in higher Sobolev norms will be useful when we want to infer stabilization not only in $H^1(0,L)$, but also in  $H^m(0,L)$, for $m\geq2$, see Theorem~\ref{t-5.6} below.

Let us call the semigroup of weak solutions generated by \eqref{eq2} by $S_\beta$\,.
Here is our basic result on $S_\beta$\,. 
\begin{proposition} [Absorbing ball in $ H^1(0,L)$]\label{H1absorb1D} 
Let us suppose that $\beta\in(0,1]$. The semigroup $S_\beta(t) \colon V \rightarrow V$, $u_0 \mapsto S_\beta(t)u_0 = u(t)$, generated by  the unique weak solutions to equation to (\ref{eq2}) for $u_0\in V$ has an $H^1$ absorbing ball $\cB = \{u \in V : \|u\|_{ H^1(0,L)} \leq M_1(L,\zeta)\}$, \emph{i.e.}, for a set $B\subset  V$ bounded in the $ H^1$-topology, there is $t_B> 0$ such that $S(t)u_0 =u(t) \in \cB$ for $u_0\in B$ and $t\ge t_B$\,. 
The value of the constant $M_1(L,\zeta)>0$ is provided at the end of the proof.
\end{proposition}

\begin{proof} We will proceed in a number of steps developing a priori estimates.

\noindent{\it Step 1.} We claim that weak solutions, in the sense of Definition \ref{df-weak}, to equation (\ref{eq2}) with $u_0\in V$ 
fulfill, for every $\beta\leq1$,
\begin{equation}\label{eqn:1st_est}
 \frac{\de}{\de t} \bigg[ \int_{0}^{L} W_0(u+c_0)\, \de x + \frac{1}{2} \| u_x \|^2  \bigg]+ \frac{1}{2} \| (-\Delta)^{-1/2} u_t \|^2  \leq  \beta\|u\|^2 + \norm{\zeta_x}^2,
\end{equation}
where we recall that $W_0$ is the classical double-well potential defined in \eqref{df-W}.
Indeed, application of the integral operator $(-\Delta)^{-1} \colon L^2 \to  D(-\Delta)\subset H^2$ to both sides of equation \eqref{eq2}    
yields 
\begin{equation}
(-\Delta)^{-1} u_t - u_{xx} + (u+c_0)^3 - (u+c_0) +\zeta -
\beta (-\Delta)^{-1} u_x      = 0 \,,  \label{eqn:hcch2} 
\end{equation}
Since $(-\Delta)^{-1} u_t$ and $u_{xx}$ belong to $H^1$, we may pair the above equation with $u_t$\,. 
Next, integration by parts and  rearranging yield 
\begin{equation}\label{ineq0}
\begin{split}
\frac{\de}{\de t} \left[ \frac{1}{4} \| u+c_0 \|_{L^4}^4  - \frac{1}{2} \| u+c_0 \|^2 + \frac{1}{2} \| u_x \|^2  \right]  + \| (-\Delta)^{-1/2} u_t \|^2 \leq  
\beta |\dup{(-\Delta)^{-1} u_x}{ u_t}|+|\dup{\zeta}{u_t} &\,|, 
\end{split}
\end{equation}
where $\langle\cdot,\cdot\rangle$ denotes the duality pairing between $V$ and $V'$.
We notice that the right-hand side can be handled by using the Cauchy inequality with~$\epsilon=1/2$; indeed,
\begin{subequations}\label{ineq_proof}
\begin{equation}\label{ineq1}
\begin{split}
|\dup{(-\Delta)^{-1} u_x}{ u_t}| = | \dup{(-\Delta)^{-1/2} u_x}{(-\Delta)^{-1/2} u_t}| &\le
\frac{1}{4} \norm{(-\Delta)^{-1/2} u_t}^2 + \norm{(-\Delta)^{-1/2} u_x}^2\\
&\le \frac1{4} \norm{(-\Delta)^{-1/2} u_t}^2 + \norm{ u}^2,
\end{split}
\end{equation}
the last inequality being justified by having $\|(-\Delta)^{1/2}\|=1$, which follows from \eqref{r-pol}, and the fact that the norm $u\mapsto \|u_x\|$ on $V$ is equivalent to the standard one in $H^1$.
In an analogous fashion, we estimate
\begin{equation}\label{ineq2}
|\dup{\zeta}{  u_t}| = |  \dup{(-\Delta)^{1/2}\zeta}{ (-\Delta)^{-1/2} u_t}| \le \frac 14 \norm{(-\Delta)^{-1/2} u_t}^2+ \norm{\zeta_x}^2.
\end{equation}
\end{subequations}
By using inequalities \eqref{ineq_proof} in \eqref{ineq0} and taking into account that $\beta\le 1$, we obtain \eqref{eqn:1st_est}. 

\smallskip

\noindent {\it Step 2.}
Estimate \eqref{eqn:1st_est} was a preliminary result. 
The next one will give the first uniform-in-time estimate. 
We claim that the inequality  
\begin{equation}\label{second}
\frac{\de}{\de t}\|(-\Delta)^{-1/2} u\|^2 + \frac{1}{8} \|u+c_0\|_{L^4}^4 + \|u_x\|^2 \leq 2K_1(L) \,
\end{equation}
holds true for all weak solutions $u$ of \eqref{eq2} with $u_0\in V$, where $K_1(L)$ is a positive constant whose explicit expression is given by formula \eqref{r-C2} below.

Indeed, by using $u$ itself as a test function in \eqref{eqn:hcch2}, we obtain
\begin{equation}\label{r-l33}
\begin{split}
0=&\,\dup{(-\Delta)^{-1} u_t}{u}- 
\dup{\Delta u}{ u} + \dup{(u+c_0)^3}{u} - \dup{u+c_0} u +\dup\zeta u - \beta \dup{(-\Delta)^{-1}u_x} u \\
=&\,\dup{(-\Delta)^{-1/2} u_t}{(-\Delta)^{-1/2}u}+\norm{u_x}^2+\norm{u+c_0}^4_{L^4}-\norm{u+c_0}^2+\dup{\zeta}{u+c_0} \\
&\, -\dup{(u+c_0)^3}{c_0}+\dup{u+c_0}{c_0}-\dup{\zeta}{c_0}- \beta \dup{(-\Delta)^{-1}u_x}{u}.
\end{split}
\end{equation}
This is legitimate, because all the ingredients  are in $L^2$. Subsequently, we use a series of estimates based on Young's inequality and on the fact that $|c_0|\leq1$, namely,
\begin{equation}\label{estimates}
\begin{split}
c_0(u+ c_0)^3 \le &\, \frac 18(u+c_0)^4 + 54,
\qquad (u+c_0)^2\le \frac 12(u+c_0)^4 + \frac 12, \\
c_0(u+c_0) \le &\, \frac12 (c_0^2 + (u+c_0)^2)\le
\frac 12 c_0^2 + \frac 14 (u+c_0)^4 + \frac 14 \leq \frac 14 (u+c_0)^4+\frac{3}{4}.
\end{split}
\end{equation}
If we combine these inequalities with \eqref{r-l33} and take into account that $|c_0| <1$, then we reach
\begin{equation}\label{temp0}
\begin{split}
&\, \frac12 \frac{\de}{\de t}\norm{(-\Delta)^{-1/2} u}^2 + \norm{ u_x}^2 + \norm{u+c_0}_{L^4}^4 - \frac12\norm{u+c_0}_{L^4}^4-\frac{L}2 \\
\leq&\, \frac18\norm{u+c_0}_{L^4}^4+54L+\frac14\norm{u+c_0}_{L^4}^4+\frac{3L}4 + |\dup{\zeta}{c_0}|+|\dup{\zeta}{u+c_0}| +\beta |\dup{(-\Delta)^{-1}u_x}{u}|.
\end{split}
\end{equation}
We can now use Young's inequality with $\epsilon=1/8$ and the fact that $|\zeta|\leq1$ to estimate 
$$\zeta(u+c_0)\leq \frac12 |\zeta|^2+\frac12(u+c_0)^2
\leq \frac12 
+\frac12\bigg(\frac{\epsilon}{2}(u+c_0)^4+\frac1{2\epsilon}\bigg)
=\frac1{32}(u+c_0)^4+\frac{5}{2},$$
whence
\begin{equation}\label{temp1}
|\dup{\zeta}{u+c_0}|\leq \frac1{32}\norm{u+c_0}_{L^4}^4+\frac{5L}{2}.
\end{equation}
Moreover, by observing that 
\begin{equation}\label{norm_inv_laplacian}
\norm{(-\Delta)^{-1}f}\leq L^2\norm{f}
\end{equation}
and using Young's inequality with $\epsilon$ and $\eta$, we can estimate,
\begin{equation}\label{temp2}
\begin{split}
|\dup{(-\Delta)^{-1} u_x} u|= &\,|\dup{(-\Delta)^{-1} u_x}{u+c_0} - \dup{(-\Delta)^{-1} u_x}{c_0}| \\
\leq&\, \frac\epsilon2\norm{(-\Delta)^{-1}u_x}^2+\frac1{2\epsilon}\norm{u+c_0}^2+\frac\epsilon2\norm{(-\Delta)^{-1}u_x}^2+\frac1{2\epsilon}\norm{c_0}^2\\
\leq&\, L^4\epsilon\norm{u_x}^2+\frac1{2\epsilon}\frac{\eta}2\norm{u+c_0}^4_{L^4}+\frac1{2\epsilon}\frac1{2\eta}L+\frac1{2\epsilon}L \\
= &\, \frac12\norm{u_x}^2+\frac1{32}\norm{u+c_0}^4_{L^4}+8L^9+L^5,
\end{split}
\end{equation}
after choosing $\epsilon=1/2L^4$ and $\eta=1/16L^4$.
By using inequalities \eqref{temp1} and \eqref{temp2} in \eqref{temp0} and rearranging terms, we get
\begin{equation}\label{r-C2}
\begin{split}
\frac12\frac{\de}{\de t}\norm{(-\Delta)^{-1/2}u}^2 + \frac12 \norm{u_x}^2 + \frac1{16}\norm{u+c_0}_{L^4}^4 \leq &\, \frac{231}{4}L+L^5+8L^9 + |\dup{\zeta}{c_0}| \\
\leq &\, \frac{235}{4}L+L^5+8L^9 \eqqcolon K_1(L),
\end{split}
\end{equation}
which proves \eqref{second}.

\smallskip

\noindent{\it Step 3.}
Now we are able to prove the existence of an absorbing set. 
To do so, we define the ``energy'' 
\begin{equation}\label{E1}
\cE_1(t) \coloneqq \int_\Omega W_{0}(u+c_0)\, \de x + \frac{1}{2}\|u_x\|^2 +  
\|(-\Delta)^{-1/2} u\|^2. 
\end{equation}
We will prove that $\cE_1(t)$ is bounded if $t$ is larger than a certain $t_B>0$ that only depends on the data of the problem.

By adding estimate \eqref{second} to estimate \eqref{eqn:1st_est}, we obtain 
\begin{align*}
\frac{\de}{\de t} \cE_1(t) & + \theta \cE_1(t) - \theta \biggl(\int_\Omega W_{0}(u+c_0)\, \de x + \frac{1}{2}\|u_x\|^2 +  
\|(-\Delta)^{-1/2} u\|^2 \biggr) \\
& + \frac 18 
\norm{u+c_0}_{L^4}^4 + 
\norm{u_x}^2 +\frac12\|(-\Delta)^{-1/2}u_t\|^2\leq \beta \norm{u}^2 + \|\zeta_x\|^2 + 2K_1(L).
\end{align*}
Here we added and subtracted a small fraction of $\cE_1$ ($\theta>0$). 
A rearrangement yields
\begin{align*}
&\, \frac{\de}{\de t} \cE_1(t) + \theta \cE_1(t)  + \bigg(\frac18-\frac{\theta}{4}\bigg)\norm{u+c_0}_{L^4}^4 + \bigg(1-\frac{\theta}{2}\bigg) \norm{u_x}^2  \\ 
\leq &\,  \theta \|(-\Delta)^{-1} u\|^2 + \frac{L\theta}{4}
+\beta\norm{u}^2+\norm{\zeta_x}^2+ 2K_1(L)
\leq (\theta L^4 +1)\norm{u}^2+\frac{L\theta}4+\norm{\zeta_x}^2+2K_1(L),
\end{align*}
where we used \eqref{norm_inv_laplacian} and $\beta\leq1$.
By Young's inequality with $\epsilon=1/16$, we obtain
$$\norm{u}^2\leq \frac{1}{16}\norm{u+c_0}_{L^4}^4+18L.$$
Thus, recalling 
\eqref{norm_inv_laplacian}, we estimate
\begin{align*}
&\, \frac{\de}{\de t} \cE_1(t) + \theta \cE_1(t)  + \bigg(\frac18-\frac{\theta}{4}\bigg) \norm{u+c_0}_{L^4}^4 + \bigg(1-\frac{\theta}{2}\bigg) \norm{u_x}^2  \\ 
\leq &\, \frac{\theta L^4+1}{16} \norm{u+c_0}_{L^4}^4 + \biggl[18(\theta L^4+1)+\frac\theta4\biggr]L +\norm{\zeta_x}^2+2K_1(L)
\end{align*}
By choosing $\bar\theta=(L^4+4)^{-1}$, the term with $\norm{u+c_0}_{L^4}^4$ vanishes and we can neglect the term containing $\norm{u_x}^2$, since $\bar\theta<2$; we obtain
$$
\frac{\de}{\de t}\cE_1(t)+\frac1{L^4+4}\cE_1(t) \leq \frac{L(144(L^4+2)+1)}{4(L^4+4)}+\norm{\zeta_x}^2+2K_1(L)\eqqcolon K_2(L,\zeta),
$$
so that Gronwall inequality leads to
\begin{equation}\label{E1_decrease}
 \cE_1(t) \leq \bigl(\cE_1(0) -  (L^4+4) K_2(L,\zeta)  \bigr) e^{- t/(L^4+4)} + (L^4+4) K_2(L,\zeta)  
\end{equation}
and this is enough to conclude. 
Indeed, it is easy to see that there exists $t_B>0$ depending on $\cE_1(0)$, $L$, and $\zeta$ such that $\cE_1(t)<(L^4+4)K_2(L,\zeta)+1\eqqcolon M_1(L,\zeta)$ for all $t\ge t_B$\,.
\end{proof}

The proof we presented above shows a uniform-in-$\beta\le 1$ bound on the absorbing set and time~$t_B$\,. 
However, we need uniform estimates on $u$ in higher Sobolev norm. The constant variation formula \eqref{r-convar} yields bounds on $w = u-\eta$, hence we have them for $u$ too.

\begin{lemma}\label{l-H2}
Let us suppose that $w_0\in B\subset V$, where $B$ is bounded, and that $t\ge t_2 \coloneqq t_B+1$, with~$t_B$ provided by Proposition~\ref{H1absorb1D}. 
Then there exists a positive constant $M_2(B,L,\eta,\zeta)$, depending only on $B$, $L$,  $\eta$, and $\zeta$, such that $S_\beta (t)w_0 \in H^2(0,L)$ and $\|w(t)\|_{H^2(0,L)}\le M_2(B,L,\eta,\zeta)$.
\end{lemma}
\begin{proof} 
It is enough to apply $\Delta$ to both sides of \eqref{r-convar} to estimate the second-order derivative of $w$. 
If we take $\lambda \in (0, \inf \sigma(\Delta^2))$, then estimate \eqref{r-spec} yields 
\begin{align*}
\| \Delta w(t)\|
\le &\, \| \Delta( e^{-\Delta^2 (t-t_{B})}(w-\eta) + \eta)\| \\ 
& +\int_{t_{B}}^t \bigl\| \Delta\bigl( e^{-\Delta^2 (t-\sigma)}\bigl( \Delta \partial_s W(x, w + \eta + c_0) + \beta w_x - \beta \eta_x\bigr)\bigr)\bigr\|\, 
\de\sigma
\eqqcolon  I_1(t) + I_2(t).
\end{align*}
It is easy to see that due to \eqref{E1_decrease} we have the following estimate, where we used \eqref{r-spec} again, 
\begin{align*}
I_1(t) &\le C_{1/4} (t-t_{B})^{-1/4} e^{-\lambda (t-t_{B})}\| w_0\|_{H^1} + C_{1/2} (t-t_{B})^{-1/2}e^{-\lambda (t-t_{B})}\|\eta\|
+ \|\eta_{xx}\| \\ 
& \le C_{1/4} \| w_0\|_{H^1} + C_{1/2} \|\eta\|
+ \|\eta_{xx}\| \eqqcolon c_0(B,L,\eta), 
\end{align*}
for every $t\ge t_2= t_B+ 1$. Moreover,
\begin{align*}
I_2(t) \le &\, \int_{t_B}^{\infty}\| \Delta( e^{-\Delta^2 (t-\sigma)} \Delta \partial_s W(x, w + \eta + c_0)\big) \| 
\,\de\sigma 
+ \int_{t_B}^{\infty}\| \Delta\bigl( e^{-\Delta^2 (t-\sigma)}(\beta w_x - \beta \eta_x )\bigr)\|
\,\de\sigma\\
\eqqcolon &\,  J_1(t) + J_2(t).
\end{align*}
Again, since for $t\ge t_B$ the solution lies in the absorbing set, it is easy to see that
$$
J_2(t) \le \beta  \int_{t_B}^{\infty} C_{1/2} \frac{e^{-\lambda (t-s)}}{(t-s)^{1/2}} (M_1(L,\zeta)
+ \|\eta_x\|
)\,\de\sigma \eqqcolon  c_1(B,L,\eta,\zeta) .
$$
Finally, we estimate $J_1$\,; for this purpose we notice that 
$$
\big\| \big((w+\eta+ c_0)^3\big)_x\big\|
\le 3\big(LM_1(L,\zeta) 
+(\|\eta\|_{L^\infty}+ 1)^2\big) (\| w_x\|
+\| \eta\|
) \eqqcolon c_2(B,L,\eta,\zeta),
$$
where we also used that $\|w\|_{L^\infty}\le L^{1/2}\| w_x\|$ for $w\in V$.

Our claim follows by combining these estimates with the bound in $V$, {\it i.e.}, 
$$
M_2(B,L,\eta,\zeta) \coloneqq 
\sqrt{M_1(L,\zeta)^2 + c_0(B,L,\eta)^2 +
c_1(B,L,\eta,\zeta)^2 +c_2(B,L,\eta,\zeta)^2}.
$$
The proof is concluded.
\end{proof}

The above lemma implies that the absorbing set in $V$ is compact. We combine it with the following well-knows fact.
\begin{proposition} (see \cite[Theorem 1]{lukasz})\label{p-at}
Let us suppose that  a strongly  continuous semigroup $S(\cdot)$ on $Z$
has a compact attracting set $K$. Then there is a compact global attractor for $S(\cdot)$ and $\cA = \omega(K)$.
\end{proposition}
In this way we deduce the following corollary.
\begin{corollary}\label{c-at}
For all $\beta\in (0,1)$ the semigroup $S_\beta$ has a global attractor, $\cA_\beta$\,, which is compact in $H^1$-topology. Moreover, there exists $R>0$ such that $\cA_\beta$ is contained in the ball $B_V(0;R)$ in the $H^1$-topology. 
\end{corollary}
\begin{proof}
Proposition \ref{H1absorb1D} combined with Lemma \ref{l-H2} imply that the absorbing set is compact in the  $H^1$-topology. The fact that $S_\beta$ is a strongly continuous semigroup in $V$ was established in the course of proof of \cite[Theorem 6.5]{BonDavMor2018}. Hence, Proposition \ref{p-at} yields existence of a global attractor for each $\beta\in (0,1].$ The uniform bound on the attractor in  $H^1$-topology follows from Proposition~\ref{H1absorb1D}.
\end{proof}

We finally note the following uniform bounds, away from initial time.
\begin{proposition}[Uniform higher-order bounds after an arbitrary delay]\label{p-hk-gw}
Let $B\subset V$ be bounded and let $t_B$ be the entering time of the absorbing
ball provided by Proposition~\ref{H1absorb1D}. Then, for every $\tau>0$ and every
$m\ge 2$, there exists a positive constant $M_{m,\tau}=M_{m,\tau}(B,L,\eta,\zeta)$ depending on $B$, $L$, $\eta$, and $\zeta$, such that, for all $\beta\in[0,1]$, all $u_0\in B$, and all $t\ge t_B+\tau$,
\[
\|S_\beta(t)u_0\|_{H^m(0,L)}\le M_{m,\tau}.
\]
\end{proposition}
\begin{proof}
Let $w(t)=S_\beta(t)u_0-\eta$.
Fix $\tau>0$ and set
\[
\tau_j \coloneqq 2^{-j}\tau,
\qquad
s_n \coloneqq \sum_{j=1}^n \tau_j.
\]
Then $s_n<\tau$ for every $n\in\bN$ and $s_n\uparrow\tau$.

By Proposition~\ref{H1absorb1D}, there exists $M_1(L,\eta,\zeta)>0$ such that
\[
\sup_{\beta\in[0,1]}\ \sup_{u_0\in B}\ \sup_{t\ge t_B}\|w(t)\|_{H^1(0,L)}\le M_1(L,\eta,\zeta).
\]

We start with the case $m=2$.
Arguing exactly as in the proof of
Lemma~\ref{l-H2}, but writing \eqref{r-convar} on the interval
$[t-\tau_1,t]$ instead of $[t_B,t]$, we find a positive constant $M_{2,\tau}=M_{2,\tau}(\PR{B},L,\eta,\zeta)$ such
that
\[
\sup_{\beta\in[0,1]}\ \sup_{u_0\in B}\ \sup_{t\ge t_B+s_1}
\|w(t)\|_{H^2(0,L)}\le M_{2,\tau}\,.
\]
Since we have an absorbing set in $H^1$, the constant $M_{2,\tau}$ depends on $B$ not on an individual $u_0\in B.$ Neither $M_{2,\tau}$ depends on $t$, because the function $s\mapsto e^{-\lambda s} s^{-\alpha}$ is integrable over $(0,\infty)$, provided that $\lambda>0$ and $\alpha\in (0,1)$.

We now bootstrap.
Assume that for some $m\ge 2$ we already know that there exists a positive constant $M_{m,\tau}=M_{m,\tau}(u_0,L,\eta,\zeta)$ such that
\[
\sup_{\beta\in[0,1]}\ \sup_{u_0\in B}\ \sup_{t\ge t_B+s_{m-1}}
\|w(t)\|_{H^m(0,L)}\le M_{m,\tau}\,.
\]
Fix $t\ge t_B+s_m$\,. 
Then $t-\tau_m\ge t_B+s_{m-1}$\,, so the induction hypothesis is available on the whole interval $[t-\tau_m,t]$.
Arguing as in \eqref{eq_banach_algebra_etc}  we have the existence of a positive constant
\[
\sup_{s\in[t-\tau_m,t]}\|F_\beta(w(s))\|_{H^{m-2}(0,L)}
\le \widetilde{C}_{m,\tau},
\]
where $C_{m,\tau}$ is independent of $t$, $\beta$, and $u_0\in B$. 
Writing \eqref{r-convar} with initial time $t-\tau_m$, applying
$A^{(m+1)/4}$, and using \eqref{r-spec}, we obtain
\begin{align*}
\|w(t)\|_{H^{m+1}(0,L)}
\le &\, C\big\|A^{1/4}e^{-A\tau_m}A^{m/4}w(t-\tau_m)\big\| \\
&\,+ C\int_{t-\tau_m}^t
\big\|A^{3/4}e^{-A(t-s)}A^{(m-2)/4}F_\beta(w(s))\big\|\,\de s \\
\le &\, C\tau_m^{-1/4}\|w(t-\tau_m)\|_{H^m(0,L)}
 + C\int_{t-\tau_m}^t (t-s)^{-3/4}
   \|F_\beta(w(s))\|_{H^{m-2}(0,L)}\,\de s \\
\le &\, M_{m+1,\tau}. 
\end{align*}
The constant $M_{m+1,\tau}$ is independent of $t$,
$\beta$, and $u_0\in B$.

This proves by induction that for every $m\ge 2$,
\[
\sup_{\beta\in[0,1]}\ \sup_{u_0\in B}\ \sup_{t\ge t_B+s_{m-1}}
\|w(t)\|_{H^m(0,L)}\le M_{m,\tau}.
\]
Since $s_{m-1}<\tau$, the same bound holds a fortiori for all $t\ge t_B+\tau$.
Finally, since $\eta$ is smooth, the estimate transfers from $w$ to $u=w+\eta$.
\end{proof}

\section{Stabilization of solutions}\label{stabilization}
\subsection{Tools of the dynamical systems}
Our aim is to recall the weak notion of the gradient flow studied in \cite{CLR}. It is based on the notion of global compact attractor.
\begin{definition}{\rm (\cite[Definition 1.5]{CLR})}
A set $\cA \subset Z$ is the global attractor for a semigroup $S (\cdot) \colon Z \to Z$ if (i) $\cA$ is compact; (ii) $\cA$ is invariant; (iii) $\cA$ attracts each bounded subset of $Z$.
\end{definition}
We may now recall notion of the gradient  flow used in \cite{CLR}.

\begin{definition} \label{defRob}{\cite[Definition 5.3, Definition 5.4, Theorem 5.5]{CLR}}
We say that a semigroup $S$ with a global attractor $\cA$ is a gradient flow with respect to the family $\cS=\{\cE^{0},\ldots,\cE^{k}\}$ of invariant sets provided  that:\\
1) For any global (eternal) solution $\xi \colon \bR \to Z$ taking values in $\cA$, there exist $i,j\in\{0,\ldots, k\}$ such that
$$
\lim_{t\to -\infty} \dist(\xi(t), \cE^{i}) =0\quad\hbox{and}\quad
\lim_{t\to +\infty} \dist(\xi(t), \cE^{j}) =0.
$$
2) The collection $\cS$ contains no homoclinic structures.
\end{definition}

Before we state our main tool we recall that $\dist_\rmH(A,B)$ denotes the  Hausdorff distance between compact sets $A$ and $B$, which is defined with the help of the metric of the ambient space $Z.$ 

The theorem below calls for checking the collective assymptotic compactness. We say that  strongly continuous semigroups $\{S_\beta(\cdot)\}_{\beta}$ on $Z$   are {\it collectively asymptotically compact} provided that for any sequence $\{\beta_n\}_n$\, for any sequence  $t_n\to+\infty$, and any bounded sequence $\{x_n\}\subset Z$ such that  
$\{S_{\beta_n}(t_n)x_n\}_n$ is also bounded, we can show that $\{S_{\beta_n}(t_n)x_n\}_n$ contains a convergent subsequence.

Our main tool is the following theorem.
\begin{theorem}\label{thmRob} {\rm (\cite[Theorem 5.26]{CLR}: Stability of gradient semigroups).}
Let $S_0(\cdot)$ be a semigroup on a Banach space $(Z,\norm{\cdot}_Z)$ that has a global attractor
$\cA_0$ and that is a gradient flow with respect to the finite collection $\cS^0$ of isolated invariant sets $\{\cE_{0}^{0} ,\cE_{0}^{1} ,\ldots, \cE_{0}^{k}\}$. Assume that:
\begin{enumerate}
\item[(a)] for each $\beta\in(0,1]$, $S_\beta(\cdot)$ is a semigroup on $Z$ with a
global attractor $\cA_\beta$;
\item[(b)] $\{S_\beta(\cdot)\}_{\beta\in[0,1]}$ is collectively asymptotically compact and
$\overline{\bigcup_{\beta\in[0,1]}\cA_\beta}$ is bounded;
\item[(c)] $S_\beta(\cdot)$ converges to $S_0(\cdot)$, in the sense that
\[
\norm{S_\beta(t)u - S_0(t)u}_Z \rightarrow 0\qquad\text{ as }\quad\beta\rightarrow0
\]
uniformly for $u$ in compact subsets of $Z$; and
\item[(d)] for $\beta\in(0,1]$, the attractors $\cA_\beta$ contains a finite collection of isolated invariant sets $\cS_\beta = \{\cE_{\beta}^{0} ,\cE_{\beta}^{1},\ldots,\cE_{\beta}^{k}\}$ such that
\[
\lim_{\beta\to0} \dist_{\rmH}(\cE_{\beta}^{j}\,,\cE_{0}^{j}) = 0
\]
and there exist $\eta > 0$ and $\beta_1\in(0,1)$ such that for all $\beta\in(0,\beta_1)$, if $\xi_\beta \colon \bR\rightarrow \cA_\beta$ is a global (or eternal) solution, then
\[
\sup_{t\in\bR} \dist( \xi_\beta(t), \cE_{\beta}^{j})\le\eta 
\quad
\Rightarrow
\quad
\xi_\beta(t)\in \cE_{\beta}^{j}\text{ for all $t\in\bR$.}
\]
\end{enumerate}
Then there exists a $\beta_{2} \in (0,\beta_1)$ such that $\{S_\beta(\cdot)\}_{\beta\in(0,\beta_2)}$ is a gradient semigroup with respect to~$\cS_\beta$\,. 
In particular, for $\beta\in(0,\beta_2)$,
\[
\cA_\beta = \bigcup_{i=1}^k \cW^{\rmu}(\cE_{\beta}^{i})\,.
\]
\end{theorem}

One of the thing we have to do before we apply Theorem \ref{thmRob} is to choose the Banach space $Z$. 
Corollary~\ref{c-at} tells us that the right choice  for $Z$ is $Z= V$, which is equal to $\scrH^{1/4}$.

\subsection{Convergence of solutions}
In this subsection we check that conditions of Theorem \ref{thmRob} are satisfied.
\begin{lemma}[Hypothesis (b) of Theorem~\ref{thmRob}]\label{lem:part-b}
Let $Z=V=\scrH^{1/4}$. Then the family $\{S_\beta(\cdot)\}_{\beta\in[0,1]}$ is collectively asymptotically compact on $Z$.
Moreover, $\overline{\bigcup_{\beta\in[0,1]}\cA_\beta}$ is bounded in $Z$. 
\end{lemma}

\begin{proof}
\emph{Boundedness of the union of attractors.}
By Corollary~\ref{c-at}, there exists $R>0$, independent of $\beta$, such that
\[
\cA_\beta \subset B_{V}(0;R)\qquad\text{for all }\beta\in(0,1].
\]
The same uniform dissipative estimate used to prove Corollary~\ref{c-at} (and ultimately Proposition~\ref{H1absorb1D})
does not depend on~$\beta$ (the transport term is lower order), hence the same~$R$ also bounds~$\cA_0$ in~$H^1(0,L)$.
Therefore, $\bigcup_{\beta\in[0,1]}\cA_\beta$ is bounded in $V$, hence so is its closure.

\smallskip
\emph{Collective asymptotic compactness.}
Let $\{\beta_n\}\subset[0,1]$, let $t_n\to+\infty$, and let $B \coloneqq \{x_n: n\in\bN\}\subset V$ be bounded.
By Lemma~\ref{l-H2}, there exist $t_2>0$ and $M_2=M_2(B,L,\eta,\zeta)>0$ (depending only on the bound of $B$ in $V$, but not on $\beta$)
such that for every $\beta\in[0,1]$ and every $x\in B$,
\[
S_\beta(t)x \in H^2(0,L)
\quad\text{and}\quad
\norm{S_\beta(t)x}_{H^2(0,L)}\le M_2
\qquad\text{for all }t\ge t_2\,.
\]
Choose $N\in\mathbb{N}$ such that $t_n\ge t_2$ for all $n\ge N$. Since, in addition, $S_{\beta_n}(t_n)x_n\in D(A)$, then $\{S_{\beta_n}(t_n)x_n\}_{n\ge N}$ is bounded in $\scrH^{1/2}\subset H^2(0,L)$.
Since the embedding $\scrH^{1/2}\hookrightarrow \scrH^{1/4}=V$ is compact by the Rellich Theorem, the sequence
$\{S_{\beta_n}(t_n)x_n\}_{n\ge N}$ has a convergent subsequence in~$V$.
This is precisely collective asymptotic compactness in $Z=V$.
\end{proof}

Now, we will check that hypothesis (c) holds.
\begin{lemma}[Hypothesis {\rm(c)} of Theorem~\ref{thmRob}]\label{part-c}
Let $Z=V=\{u\in H^1(0,L):u(0)=0\}$ and let $K\subset V$ be compact.
Then for every $T>0$ we have
\[
\sup_{t\in[0,T]}\ \sup_{u_0\in K}\ \|S_\beta(t)u_0-S_0(t)u_0\|_{Z}\ \longrightarrow\ 0
\qquad\text{as }\beta\to 0.
\]
In particular, for each fixed $t\ge 0$,
\[
\sup_{u_0\in K}\|S_\beta(t)u_0-S_0(t)u_0\|_{Z}\ \longrightarrow\ 0
\qquad\text{as }\beta\to 0,
\]
which is exactly hypothesis {\rm(c)} of Theorem~\ref{thmRob}.
\end{lemma}

\begin{proof}
Fix $T>0$. Since $K\subset V$ is compact, it is bounded, hence there exists $M>0$ such that
\[
\|u_0\|_{V}\le M\qquad\text{for all }u_0\in K.
\]
For $u_0\in K$ and $\beta\in(0,1)$, set $u_\beta(t) \coloneqq S_\beta(t)u_0$ and $\bar u(t) \coloneqq S_0(t)u_0$.
By \cite[Theorem~6.6]{BonDavMor2018} (applied with $\beta_0=1$ and with our potential $W$, see~\eqref{df-W}),
there exists a constant $C>0$, depending only on $L$, $W$, $T$, and the bound $M$ (but not on the particular choice of $u_0\in K$), such that
\[
\|u_\beta(t)-\bar u(t)\|_{V}\le C\,\beta
\qquad\text{for all }t\in[0,T]\text{ and all }\beta\in(0,1).
\]
Taking the supremum over $u_0\in K$ and $t\in[0,T]$ gives
\[
\sup_{t\in[0,T]}\ \sup_{u_0\in K}\ \|S_\beta(t)u_0-S_0(t)u_0\|_{V}
\le C\,\beta \xrightarrow[\beta\to0]{}0.
\]
Since $Z=V$, this yields the desired convergence in $Z$ and proves hypothesis {\rm(c)}.
\end{proof}

We are going to check that part (d) of the hypothesis of Theorem \ref{thmRob} holds. In the present case the invariant sets $\cE_{\beta}^{j}$ are points in $Z$, \emph{i.e.}, $u_\beta^j$\,, hence $\dist_\rmH (\cE_{\beta}^{j}\,, \cE_{0}^{j})= \| u_\beta^j - u_0^j\|_Z$\,.
\begin{lemma}[hypothesis (d) of Theorem~\ref{thmRob} is satisfied]\label{part-d} 
Let us suppose that $L>0$ does not belong to $E$. Then,
\begin{equation}\label{eq_beta_limit}
\lim_{\beta\to 0}\| u_\beta^j - u_0^j\|_Z=0.
\end{equation}
Moreover, there exist $\gamma>0$ and $\beta_1>0$ such that if $\xi_\beta \colon \bR\to\cA_\beta$\,, for all $\beta<\beta_1$\,, is an eternal solution, then
\[
\sup_{t\in\bR} \norm{\xi_\beta(t) - u_\beta^j}_Z\le\gamma 
\quad
\Longrightarrow
\quad
\xi_\beta(t) = u_\beta^j\quad \text{ for all $t\in\bR$.}
\]
\end{lemma}
\begin{proof}
{\it Step 1.} The limit in \eqref{eq_beta_limit} is an immediate consequence of \eqref{eq_small_beta_estimate} from Theorem~\ref{thm:IFT}. 
We are left with the latter statement. 
Since we deal with couples $(L, \tilde z^j)$ which are nondegenerate, as an immediate consequence, we have that the linearized operator $\scrL_L^\beta$ associated with~\eqref{eq2}, namely 
$$\scrL_L^\beta h = \Delta(-\Delta h+\partial^2_{s^2}W_0(u^j_\beta)h)-\beta h_x\,,
$$
is non-singular when $\beta =0$, due to  Theorem \ref{cor:discreteL}. 

\smallskip

{\it Step 2.} We claim that $\scrL_L^\beta$ is not only non-singular for $\beta\in(0,\beta_1)$, for a certain $\beta_1>0$ (that will be defined at the end of this step), but also that there exists $\delta_0>0$ such that for all $\delta\in (-\delta_0, \delta_0)$ and all $a\in \bR$ we have
\begin{equation}\label{r5.2}
\delta +  ai  \in \rho(\scrL_L^\beta)\qquad\hbox{for }\beta\in(0,\beta_1).
\end{equation}
For this purpose we recall a well-known fact that $\scrL_L^0$ is equivalent to a self-adjoint operator $T \colon D(T)\subset L^2(0,L)\to L^2(0,L)$, \emph{i.e.}, 
$$
(-\Delta)^{1/2}(-\Delta +\partial^2_{s^2}W_0(u^j_\beta))(-\Delta)^{1/2} \eqqcolon  T = (-\Delta)^{-1/2} \scrL_L^0 (-\Delta)^{1/2},
$$
where 
$$
D(T) =\{  u \in H^4: (-\Delta)^{1/2}u(0)=((-\Delta)^{1/2}u)_x(L) =\mu((-\Delta)^{1/2}u)(0)=\mu((-\Delta)^{1/2}u)_x(L)  =0\}.
$$
The operators $-\Delta$
and $-\Delta + \partial^2_{s^2}W_0(u^j_0+c_0)$ are self-adjoint. Hence,
\begin{align*}
(T u, w) &=  
\big( (-\Delta +\partial^2_{s^2}W_0(x, u^j_\beta+c_0))(-\Delta)^{1/2}u, (-\Delta)^{1/2} w \big) \\
& =
\big( (-\Delta)^{1/2}u, (-\Delta +\partial^2_{s^2}W_0(u^j_\beta))(-\Delta)^{1/2} w \big)
=
\big( u, (-\Delta)^{1/2}(-\Delta +\partial^2_{s^2}W_0(u^j_\beta))(-\Delta)^{1/2}  w \big) \\ 
&= (u, Tw).
\end{align*}
Due to this equivalence we conclude that the spectrum of $\scrL_L^0$ is real.

We define $\delta_0>$ as follows. Since $\sigma(\scrL_L^0)$ is real discrete and it does not contain zero there is $\delta_0>0$ such that if $z$ belongs to the strip $\Sigma_{\delta_0}=\{z\in \bC: |\Re z|\le \delta_0\} $, then the distance from $z$ to $\sigma(\scrL_L^0)$ is at least $\delta_0$. 
Now, we shall see that there is $\beta_1>0$ such that $\Sigma_{\delta_0} \subset \rho( \scrL_L^\beta)$ for all $\beta\in (0, \beta_1)$.
We compute
\begin{align*}
(  \scrL_L^\beta - (\delta+ ai) Id)^{-1} &= ( \scrL_L^0 -(\delta+ai) Id+  \scrL_L^\beta -  \scrL_L^0)^{-1}\\& =  
( \scrL_L^0 -(\delta+ai) Id)^{-1}( Id + ( \scrL_L^0 -(\delta +ai) Id)^{-1}(\scrL_L^\beta -  \scrL_L^0))^{-1}.
\end{align*}
Our claim will follow if we show that
$$
\| ( \scrL_L^0 -(\delta +ai) Id)^{-1}(\scrL_L^\beta -  \scrL_L^0)\| \le \frac12
$$
for all $\beta\in (0, \beta_1)$ independently of $a$ and $\delta_0$\,. 
We notice that
$$
\| ( \scrL_L^0 -(\delta+ai) Id)^{-1}(\scrL_L^\beta -  \scrL_L^0)\| \le 
\| ( \scrL_L^0 -(\delta+ai) Id)^{-1} \scrL_L^0 \|\cdot \|(\scrL_L^0)^{-1}(\scrL_L^\beta -  \scrL_L^0)\| .
$$
Moreover, since
$$
( \scrL_L^0 -(\delta+ai) Id)^{-1} \scrL_L^0 = Id +(\delta+ai) ( \scrL_L^0 -ai Id)^{-1} 
$$
and $\|  ( \scrL_L^0 -(\delta+ai) Id)^{-1}\|$ is bounded by the inverse of the distance from  $\delta+ai$ to $\sigma( \scrL_L^0)$, we infer that
$$
\| ( \scrL_L^0 -(\delta+ai) Id)^{-1} \scrL_L^0 \| \le 1 + \frac{\sqrt{\delta_0^2+ a^2}}{\sqrt{\delta_0^2+ a^2}} = 2.
$$
Now, we investigate $(\scrL_L^0)^{-1}(\scrL_L^\beta -  \scrL_L^0).$ We notice that
$$
(\scrL_L^\beta -  \scrL_L^0)h = 3\Delta[( (u^j_\beta)^2 - (u^j_0)^2 )h] - \beta h_x\,.
$$
Now, Theorem \ref{thm:IFT} implies that 
$$
\| (\scrL_L^0)^{-1}(\scrL_L^\beta -  \scrL_L^0)\| \le C_1 \beta \qquad\hbox{for }\beta\in (0, \beta_0)
$$
where $C_1$ depends on $\|u^j_0\|_{H^4}$ and the constant $C$ appearing in \eqref{eq_small_beta_estimate}.
By combining these estimates, we reach
$$
\| ( \scrL_L^0 -ai Id)^{-1}(\scrL_L^\beta -  \scrL_L^0)\| \le 2C_1 \beta \le
\frac{\beta_0}2\leq\frac12,
$$
for $0<\beta < \beta_1 \coloneqq \frac{\beta_0}{4C_1}$.

\smallskip

{\it Step 3.} By \eqref{r5.2}, the equilibrium $u_\beta^j$ is hyperbolic. Hence there exist local stable and unstable manifolds
$\cW^{\rms}_{\rm loc}(u_\beta^j)$ and $\cW^{\rmu}_{\rm loc}(u_\beta^j)$ (see \cite[Theorem~2.3]{hale}).
Moreover, there exists $\gamma>0$ such that the following characterisation holds:
if a solution $z(t)$ satisfies $z(0)\in B_Z(u_\beta^j,\gamma)$ and $z(t)\in B_Z(u_\beta^j,\gamma)$ for all $t\ge 0$,
then $z(0)\in \cW^{\rms}_{\rm loc}(u_\beta^j)$; similarly, if $z(t)\in B_Z(u_\beta^j,\gamma)$ for all $t\le 0$,
then $z(0)\in \cW^{\rmu}_{\rm loc}(u_\beta^j)$.

Now let $\xi_\beta \colon \bR\to \cA_\beta$ be an eternal solution with
$\sup_{t\in\bR}\|\xi_\beta(t)-u_\beta^j\|_Z\le \gamma$.
Then $\xi_\beta(0)\in \cW^{\rms}_{\rm loc}(u_\beta^j)$ (by the $t\ge0$ condition) and
$\xi_\beta(0)\in \cW^{\rmu}_{\rm loc}(u_\beta^j)$ (by the $t\le0$ condition), hence
\[
\xi_\beta(0)\in \cW^{\rms}_{\rm loc}(u_\beta^j)\cap \cW^{\rmu}_{\rm loc}(u_\beta^j)=\{u_\beta^j\}.
\]
Therefore $\xi_\beta(t)\equiv u_\beta^j$ for all $t\in\bR$, which is exactly the implication required in hypothesis~(d).
\end{proof}

After checking that the assumptions of Theorem \ref{thmRob} are satisfied we may state our main result.
\begin{theorem}\label{t-5.6}
Let us suppose that $L>0$ does not belong to $E$ and let $\beta\in (0, \beta_2)$, where $\beta_2$ is provided by Theorem~\ref{thmRob} for this $L$.
Then for any $u_0\in V$ there exists $\psi\in C^\infty(0,L)$, such that the unique solution $u$ to \eqref{eq2} corresponding to $u_0$ converges to $\psi$ in $H^m(0,L)$ for all $m\in \bN$.
\end{theorem}
\begin{proof}
We divide the proof into three steps.

{\it Step 1.} We are going to check that we may invoke Theorem \ref{thmRob} while taking $Z =V = \scrH^{1/4}$. In the course of proof of Corollary  \ref{c-at} we recalled that semigroup $S_\beta$ generated by \eqref{eq2} is strongly continuous in $V$. 
Above all, this corollary shows existence of global attractors and a uniform bound, hence the hypothesis (a) of Theorem \ref{thmRob} holds.

\smallskip
{\it Step 2.} We showed in Lemmas \ref{lem:part-b}, \ref{part-c}, and \ref{part-d} that the assumptions (b), (c), and (d) of Theorem \ref{thmRob} are satisfied. Thus, 
we deduce from this result existence of positive $\beta_2$ such 
that each semigroup $S_\beta$ for $\beta\in (0, \beta_2)$ is a gradient flow in the sense of Definition \ref{defRob}.  This implies that for any $u_0\in V$ the $\omega$-limit set $\omega(u_0)$ may consist only of the steady states, because there are no homoclinic orbits. 
Since the set $\cS_\beta$ of steady states is finite and $\omega(u_0)$ is connected, we see that it is a singleton, thus there exists $\psi\in \scrH^{1/4}$ such that
$$
\lim_{t\to+\infty} \| u(t) - \psi\|_V =0.
$$

\smallskip
{\it Step 3.}
By Lemma~\ref{lem:instantaneous_smoothing}, we have that $u(t)\in C^\infty(0,L)$ for every $t>0$.
Moreover, Proposition~\ref{p-hk-gw} applied with $B=\{u_0\}$ and $\tau=1$ shows that, for
every $N\ge 2$,
\[
\sup_{t\ge t_B+1}\|u(t)\|_{H^N(0,L)}<\infty.
\]
Since $\psi$ is a steady state of \eqref{eq2}, elliptic bootstrapping applied to
the stationary equation yields $\psi\in C^\infty(0,L)$.

We already know from Step~2 that $u(t)\to\psi$ in $V=H^1(0,L)$. Fix $m\ge 2$ and
choose $N>m$. Since both $u(t)$ and $\psi$ are bounded in $H^N(0,L)$ for
$t\ge t_B+1$, interpolation gives
\[
\|u(t)-\psi\|_{H^m(0,L)}
\le C\,\|u(t)-\psi\|_{H^1(0,L)}^\theta
       \|u(t)-\psi\|_{H^N(0,L)}^{1-\theta}
\to 0
\qquad\text{as }t\to\infty
\]
for a suitable $\theta\in(0,1)$. This proves convergence in $H^m(0,L)$ for every
$m\in \bN$.
\end{proof}

\section*{Acknowledgment}
This project  got off the ground when all authors attended the MATRIX program \emph{Gradient flows in Geometry and PDE}. The authors would like to thank MATRIX for providing an outstanding research environment and acknowledge their generous hospitality.
MM is a member of INdAM-GNAMPA and thanks IDUB for partially supporting his visit to the University of Warsaw, where a part of the work was done. PR thanks INdAM and Politecnico di Torino for supporting his visit to Politecnico di Torino, where a part of the work was done.
GW gratefully acknowledges partial financial support by ARC Discovery Project DP250101080.

\bibliographystyle{siam}
\bibliography{MRW}
\end{document}